\documentclass[12pt,reqno]{amsart}
\usepackage[utf8]{inputenc}
\usepackage{amsmath,mathtools}
\usepackage{amsthm}
\usepackage{amssymb}
\usepackage{comment}
\usepackage[hidelinks]{hyperref}
\usepackage[initials,msc-links]{amsrefs}

\usepackage{fullpage}
\usepackage{todonotes}
\newtheorem{theorem}{Theorem}[section]

\newtheorem{lemma}[theorem]{Lemma}
\newtheorem{corollary}[theorem]{Corollary}
\theoremstyle{remark}
\newtheorem{remark}[theorem]{Remark}

\newcommand\Mpax{M^{(\alpha)}_{p}(x)}
\newcommand\Mpasx{M^{(\alpha)*}_{p}(x)}
\newcommand\Npax{N_{p}^{(\alpha)}(x)}
\newcommand\Npatilx{\widetilde{N}^{(\alpha)}_{p}(x)}
\newcommand\MpaxInt{\int_0^\lambda \Mpax  \, \mathrm d\alpha}

\newcommand\ceilalphak{\lceil \alpha k \rceil}
\newcommand\floorOmalphak{\lfloor (1-\alpha) k \rfloor}
\newcommand\freGamxi{\frac{e^{-\gamma \xi}}{\Gamma(1+\xi)}}
\newcommand\Npox{N_{p, 1}(x)}

\newcommand\Npaox{N^{(\alpha)}_{p, 1}(x)}
\newcommand\Npatx{N^{(\alpha)}_{p, 2}(x)}
\newcommand\OOk{\left(1+O\left(\frac1k\right)\right)}
\newcommand\frpOmalk{\{(1-\alpha)k\}}
\newcommand\chiexpfr{\chi^{\frpOmalk}}
\newcommand\nuexpfr{\nu^{\frpOmalk}}
\newcommand\alphbetcom{\left(\frac\beta\alpha\right)^\alpha\left(\frac{1-\beta}{1-\alpha}\right)^{1-\alpha}}
\newcommand\TDelPar{\left(\frac{2-\delta}\alpha\right)}
\newcommand\TDelAlp{\left(\frac{2-\delta}\alpha\right) \log_2 p}
\newcommand\betalp{\beta_\alpha}
\newcommand\betalpPl{\betalp + \epsilon}
\newcommand\betalpM{\betalp - \epsilon}

\newcommand\SmallBetaExp{1-2\beta-\nu^\alpha(1-\beta)}
\newcommand\dalp{\, \mathrm d\alpha}
\newcommand\deta{\, \mathrm d\eta}
\newcommand\err{\mathcal E}
\newcommand\epse{\epsilon/8}

\newcommand\epsf{\frac {\epsilon}{4}}
\newcommand\epsff{\frac\epsilon4}
\newcommand\nuepse{\nu_{\epse}}
\newcommand\betaepse{\beta_{\epse}}
\newcommand\explgxeps{\exp((\log x)^\epsilon)}
\newcommand\explgxoeps{\exp((\log x)^{1-\epsilon})}

\title{The distribution of intermediate prime factors}

\author{Nathan McNew}
\email{nmcnew@towson.edu}

\address{Department of Mathematics \\ Towson University \\ Towson, MD 21252}

\author{Paul Pollack}
\email{pollack@uga.edu}

\author{Akash Singha Roy}
\email{akash01s.roy@gmail.com}

\address{Department of Mathematics \\ University of Georgia \\ Athens, GA 30602}

\thanks{P.P. is supported by the NSF under award DMS-2001581.}


\begin{document}

\begin{abstract}
    Let $P^{\left(\frac 12\right)}(n)$ denote the middle prime factor of $n$ (taking into account multiplicity). More generally, one can consider, for any $\alpha \in (0,1)$, the $\alpha$-positioned prime factor of $n$, $P^{(\alpha)}(n)$.  It has previously been shown that $\log \log P^{(\alpha)}(n)$ has normal order $\alpha \log \log x$, and its values follow a Gaussian distribution around this value. We extend this work by obtaining an asymptotic formula for the count of $n\leq x$ for which $P^{(\alpha)}(n)=p$, for primes $p$ in a wide range up to $x$.  We give several applications of these results, including an exploration of the geometric mean of the middle prime factors, for which we find that $\frac 1x \sum_{1<n \le x} \log P^{\left(\frac 12 \right)}(n) \sim  A(\log x)^{\varphi-1}$, where $\varphi$ is the golden ratio, and $A$ is an explicit constant. Along the way, we obtain an extension of Lichtman's recent work on the ``dissected'' Mertens' theorem sums $\sum_{\substack{P^+(n) \le y \\ \Omega(n)=k}} \frac{1}{n}$ for large values of $k$.
\end{abstract}

\maketitle

\section{Introduction}

Starting with \cite{DKLuca13} various papers have considered the distributional properties of the ``middle prime factor'' of an integer.  Suppose the prime factorization of an integer $n>1$ is written as \[n=q_1^{a_1}q_2^{a_2} \cdots q_{\omega(n)}^{a_{\omega(n)}}=p_1p_2 \cdots p_{\Omega(n)},\]
with $q_1< q_2 < \cdots <q_{\omega(n)}$ and $p_1 \leq p_2 \leq \cdots \leq p_{\Omega(n)}$. Call $Q^{(\frac 12)}(n) \coloneqq q_{\lceil \omega(n)/2 \rceil}$ the middle prime factor of $n$ without considering multiplicity and $P^{(\frac 12)}(n) \coloneqq p_{\lceil \Omega(n)/2 \rceil}$ the middle prime factor of $n$ considering multiplicity. 

The study of the middle prime factor without considering multiplicity was first taken up in \cite{DKLuca13}, 
where the authors obtained an asymptotic for the reciprocal sum of the middle prime factors of $n$ up to $x$, 
\[\sum_{1<n\leq x} \frac{1}{Q^{(\frac 12)}(n)} = \frac x{\log x} \exp\left((1+o(1))\sqrt{2\log_2 x \log_3 x}\right).\]
Here and throughout this paper we write $\log_k x$ to denote the $k$-fold iterated natural logarithm.  The above asymptotic was significantly sharpened by Ouellet in \cite{Ouellet17}; that paper also considered the problem generalized to the $\alpha$-positioned prime factor $Q^{(\alpha)}(n) \coloneqq q_{\lceil \alpha \omega(n)\rceil}$. In the remainder, we similarly define\footnote{As in our prior paper \cite{MPS}, we use a slightly different definition of the $\alpha$-positioned prime factor than that given before by De Koninck, Doyon, and Ouellet. 
 They take the prime factor in position $\max\{1,\lfloor \alpha (\Omega(n)+1)\rfloor\}$ rather than position $\lceil \alpha \Omega(n) \rceil$  as we do here.  These two definitions often give the same position and never differ by more than 1, and the results quoted here from other papers go through with either definition.} 
$P^{(\alpha)}(n) \coloneqq p_{\lceil \alpha \Omega(n)\rceil}$. (For convenience, we henceforth define $P^{(\alpha)}(1) \coloneqq 1$.) The asymptotic formula for the reciprocal sum of the middle prime factors considering multiplicity \cite{DoyOue18} is surprisingly different, \[\sum_{n\leq x} \frac{1}{P^{(\frac 12)} (n)} = c \frac{x}{\sqrt{\log x}}\left(1 + O\left(\frac{1}{\log_2 x}\right)\right)\]
for an explicit constant $c=\displaystyle{\frac{9}{4\sqrt\pi}\sum_p \frac{1}{p^2-2p}\prod_{3\leq q<p}\frac{1-\frac{1}{2q}}{1-\frac{2}{q}}\prod_{q}\frac{(1-\frac{1}{q})^{\frac 12}}{1-\frac{1}{2q}}}$. 

Despite their different reciprocal sums, both $\log_2 Q^{(\alpha)}(n)$ and $\log_2 P^{(\alpha)}(n)$ have normal order $\alpha \log_2 x$ (see \cite{DKK14}) and both of these values are distributed according to the Gaussian law. In particular \cite{DKDO19} shows for fixed $\epsilon \in (0,\frac 18)$, $\alpha \in (0,1)$ and $|t| \ll (\log_2 x)^\epsilon$ that \begin{equation}\frac 1x \#\left\{n \le x: \frac {\log_2 P^{(\alpha)}(n)-\alpha\log_2 x}{\sqrt{\log_2 x}} < t\right\} = \Phi\left(\frac{t}{\sqrt{\alpha(1-\alpha)}}\right) + O_\alpha\left(\frac{1}{\sqrt{\log_3 x}}\right)\label{eq:normal}\end{equation}
where $\Phi(\tau) = \frac{1}{\sqrt{2\pi}} \int_{-\infty}^\tau e^{-v^2/2} dv$ is the Gaussian probability distribution function. 

Recently in \cite{MPS}, the authors of this paper studied the distribution of the values of $P^{(\alpha)}(n)$ in a variety of settings. In particular this $\alpha$-positioned prime factor obeys Benford's leading digit law, and is equidistributed in coprime residue classes modulo $q$ where $q$ can be nearly as large as, but not significantly larger than, $(\log x)^{c(\alpha)}$, for $c(\alpha) \coloneqq 1-2^{-\alpha/(1-\alpha)}$. 

\section{Results}
In this paper we investigate the distributional properties of the function $P^{(\alpha)}(n)$ in much greater detail. In particular, we determine, for a wide range of prime numbers $p$, an asymptotic for the number of integers up to $x$ that have $p$ as the middle (or $\alpha$-positioned) prime factor. We define \[\Mpax \coloneqq \#\{n\leq x : P^{(\alpha)}(n) = p\}\] and set \[ \beta \coloneqq \frac{\log_2 p}{\log_2 x}\] (so that $\log p = (\log x)^\beta$).  Throughout the remainder of the paper, $p$ and $\beta$ will be related by this expression. We first consider the middle prime factor (the case when $\alpha = \frac 12$).

\begin{theorem} \label{thm:middle} Let $\epsilon>0$ and suppose that $p$ is a prime number, $p \to \infty$.  Then if either $\beta<\frac 15-\epsilon$ or $\frac 15 +\epsilon < \beta <1-\epsilon$ we have
\begin{equation}M^{\left( \frac 12 \right)}_p(x)= \begin{cases}\displaystyle{\left(1+O_\epsilon\left(\sqrt{\frac{\log_3 x}{\log_2 x}}\right)\right) C_\beta \displaystyle{\frac{x}{p(\log x)^{1-2\sqrt{\beta(1-\beta)}}\sqrt{\log_2 x}}}} & \text{if  }\frac{1}{5} +\epsilon < \beta <1-\epsilon,\\
\vspace{-2mm}\\
\displaystyle{\left(1+O_\epsilon\left(\sqrt{\frac{\log_3 x}{\log_2 x}} + \frac{(\log_2 p)^{-1/2}}{(\log p)^{\epsilon^2}}\right)\right) C\displaystyle{\frac{x}{p(\log x)^{\frac 12 -\frac 32 \beta}}}} & \text{if  }0 <\beta<\frac{1}{5} - \epsilon, 
\end{cases}\label{eq:mideqModif}\end{equation}
where \begin{align}
C_\beta &\coloneqq \frac{\exp\left(\frac{\gamma(1-2\beta)}{\sqrt{\beta(1-\beta)}}\right)}{\Gamma\left(1+\sqrt{\frac{\beta}{1-\beta}}\right)}\frac{\sqrt{\beta}+\sqrt{1-\beta}}{2\sqrt{\pi}\beta^{1/4}(1-\beta)^{3/4}}\prod_{q \text{ prime}}\left(1-\frac{1}{q}\right)^{\sqrt{\frac{1-\beta}{\beta}}}\left(1-\parbox{1.08cm}{$\frac{\mbox{\tiny $\sqrt{\tfrac{1-\beta}{\beta}}$}}{\mbox{\normalsize $ q$ }}$}\right)^{-1} \label{eq:Cbetaval}\\
C &\coloneqq \frac{3e^{\frac{3\gamma}{2}}}{4\sqrt{\pi}}\prod_{q> 2 \text{ prime}} \left(1+\frac{1}{q(q-2)}\right)=1.523555\ldots. \nonumber
\end{align}
\end{theorem}

While the expression for the constant $C_\beta$ above depends on $\beta$ in a complicated way, we observe that $C_{\frac{1}{2}} = \sqrt{\frac{2}{\pi}}$.

The difference in the asymptotic formulae for $M^{\left( \frac12\right)}_{p}(x)$ in the two ranges of $\beta$ considered above is largely a manifestation of the difference in the asymptotic behavior of the sum of reciprocals of $p$-smooth numbers having $k$ prime divisors, in the two cases when $k/\log_2 p$ is less than or greater than $2$ (and bounded away from $2$). For $k \le (2-\epsilon)\log_2 p$, work of Lichtman \cite{Lichtman} (Theorem \ref{thm:lichtman}) provides the desired information, while in Theorem \ref{thm:ImprExtnLicht1}, we extend Lichtman's result to the case $k \ge (2+\epsilon)\log_2 p$.

We can similarly obtain a version of this theorem that holds for general $\alpha$, however we need to introduce some additional notation before we can state this theorem.

First define the constants (depending on $\alpha$ and $\beta$) \[\chi \coloneqq \frac{(1-\alpha)\beta}{(1-\beta)\alpha}\hspace{1cm}\text{ and }\hspace{1cm}\nu \coloneqq 2^{-\frac{1}{1-\alpha}}, \] expressions which appear frequently in the statements below.  We also define $\rho_{\chi,\alpha}$ and $\rho_{\nu,\alpha}$,
 \[\rho_{c,\alpha} \coloneqq \lim_{N \to \infty} \frac{1}{N} \sum_{n=1}^N c^{\{n (1-\alpha)\}} \]
where $\{\theta\} = \theta - \lfloor \theta\rfloor$ denotes the fractional part of a real number, and $c$ will be taken to be either $\chi$ or $\nu$. Note that $\rho_{\chi, \alpha} = 1$ 
whenever $\alpha=\beta$, since in this case $\chi =1$. 

The constant $\rho_{c,\alpha}$ is the long term average of $c$ raised to the fractional part of the integer multiples of $1-\alpha$. As such, the behavior of these sums depends on whether $\alpha$ is rational or irrational.  If $1-\alpha=\frac{a}{b}$ is rational with $\gcd(a,b)=1$, we find that \[\rho_{c,\alpha}=\frac{1}{b}\sum_{i=0}^{b-1}c^{i/b} = 
\begin{cases}
\ 1 &\text{ if }c=1\\
\displaystyle{\frac{c-1}{b(c^{1/b}-1)}} &\text{ if }c \ne 1.
\end{cases}\] 
On the other hand, if $\alpha \ne \beta$ is irrational, it follows from the equidistribution theorem (see the proof of Lemma \ref{lem:beattyModified} below for more explicit statements) that \[ \rho_{c,\alpha} = \int_0^1 c^t dt = \frac{c-1}{\log c}.\] 

Finally, we define the value \begin{equation*}  \beta_\alpha \coloneqq \frac{1}{1+2^{\frac{1}{1-\alpha}}(\alpha^{-1}-1)} = \frac{1}{1+\nu^{-1}(\alpha^{-1}-1)}.\end{equation*}
As we will see below, for fixed $\alpha$, the behaviour of $\Mpax$ has a phase transition for $\beta$ on either side of this value, analogous to the transition at $\beta = \frac 15$ when $\alpha=\frac 12$. Note that $\beta_\alpha < \alpha$ for all $\alpha \in (0,1)$. 

We can now state the following generalization of Theorem \ref{thm:middle}. 

\begin{theorem} \label{thm:AlphaModified} Fix $\epsilon>0$ and $\alpha \in (0, 1)$. Suppose that $\beta \in (0, 1)$ satisfies either $0<\beta< \beta_\alpha-\epsilon$ or $\beta_\alpha +\epsilon < \beta <1-\epsilon$ and that $\alpha$ is rational. %
Then we have, as $x, p \rightarrow \infty$,
\begin{equation*}
\begin{split}
\Mpax =\!
\begin{cases}
\displaystyle{\!\left(\!1+O_{\alpha, \epsilon}\left(\sqrt{\frac{\log_3 x}{\log_2 x}}\right)\!\right) \displaystyle{\frac{C_{\beta,\alpha}x}{p(\log x)^{1-\left(\frac{\beta}{\alpha}\right)^\alpha\left(\frac{1-\beta}{1-\alpha}\right)^{1-\alpha}}\sqrt{\log_2 x}}}} & \text{if   }\beta_\alpha +\epsilon < \beta <1\!-\epsilon, \\[25pt]
\displaystyle{\!\left(\!1\!+O_{\alpha, \epsilon}\!\left(\!\! \left(\frac{\log_3 x}{\log_2 x}\right)^{\!\!\frac 14} \!{+}\, \frac{(\log_2 p)^{- \frac 12}}{(\log p)^{\epsilon^2}}\right)\!\!\right)\!\displaystyle{\frac{C_\alpha x}{p(\log x)^{ 1-2\beta-2\nu(1-\beta)}}}} %
& \text{if   }0 <\beta<\betalp - \epsilon, 
\end{cases}
\end{split}
\end{equation*}
where
\begin{align*}
C_{\beta,\alpha} &\coloneqq \frac{e^{\gamma(\chi^{\alpha-1}-\chi^\alpha)}}{\Gamma\left(1+\chi^\alpha\right)}\frac{\chi^{\frac 32\alpha-1}\rho_{\chi,\alpha}}{\sqrt{2\pi \alpha(1-\beta)}}\prod_{q \text{ prime}}\left(1-\frac{1}{q}\right)^{\chi^{\alpha-1}}\left(1-\frac{\chi^{\alpha-1}}{q}\right)^{-1}, \\
C_\alpha &\coloneqq \frac{\nu^{\alpha} \rho_{\nu,\alpha}e^{\gamma(2-\nu^\alpha)}}{2(1-\alpha)	\Gamma(1+\nu^\alpha)}\prod_{q> 2 \text{ prime}} \left(1+\frac{1}{q(q-2)}\right).
\end{align*}
The dependence of the implied constants on $\alpha$ comes from the distance of $\alpha$ from $0$ and $1$, as well as the size of the denominator of $\alpha$. 

The same asymptotic equalities hold true (without any effective error term) for each fixed irrational $\alpha$. 
For almost all $\alpha$ in this range, they hold with an explicit multiplicative error of
\begin{equation*} 
    1+O_{\alpha, \epsilon}\left(\frac{(\log_3 x)^{3/4} (\log_4 x)^{1+\epsilon}}{(\log_2 x)^{1/4}} + \frac1{(\log p)^{\epsilon_0} (\log_2 p)^{1/2}} \right). 
\end{equation*}

Finally, for any parameter $\err= o(1)$ as $x \to \infty$, %
we have, uniformly in $\beta \in (\epsilon, 1-\epsilon)$  and $\alpha \in (\beta-\err, \beta+\err)$, 
\begin{equation*}
\Mpax = \left(1+O\left(\err + \sqrt{\frac{\log_3 x}{\log_2 x}}\right)\right) \displaystyle{\frac{C_{\beta,\beta}x}{p(\log x)^{1-\left(\frac{\beta}{\alpha}\right)^\alpha\left(\frac{1-\beta}{1-\alpha}\right)^{1-\alpha}}\sqrt{\log_2 x}}}.
\end{equation*}
Here the implied constant is independent of $\alpha$. 
\end{theorem}

The obstruction to uniformity for irrational $\alpha$ comes from known bounds on the discrepancy of the Kronecker sequence $\{(1-\alpha)n\}_{n=1}^N$, see the proof of Lemma \ref{lem:beattyModified}. However, adapting some of the arguments for the above result, we can also give the following bound on $\Mpax$, which is completely uniform in all $\alpha \in (0, 1)$ and in $\beta$ away from $0$ and $1$. 
\begin{theorem}\label{thm:MpaBound}
Fix $\epsilon \in (0, 1/2)$. We have uniformly for $\alpha \in (0, 1)$ and $\beta \in (\epsilon, 1-\epsilon)$,
$$\Mpax \ll \frac x{p\sqrt{\log_2 x}}.$$
\end{theorem}  

It is worth noting that the above bound is best possible in its range of uniformity, since equality is attained in the case $\alpha = \beta$ itself (as seen from Theorem \ref{thm:AlphaModified}). Moreover, as the proof will show, the bound in Theorem \ref{thm:MpaBound} can be improved to a power saving in $\log x$, either if $\alpha$ is close to $0$ or $1$, or if $\beta$ lies in $(\delta, \beta_\alpha + \delta)$ for any $\delta$ fixed small enough in terms of $\epsilon$.  %

Because of the presence and behavior of the $\rho_{c,\alpha}$ terms, the constants $C_\alpha$ and $C_{\beta, \alpha}$ above (the latter being viewed as a function of $\alpha$ for fixed $\beta$) 
have the property of being continuous at every irrational, but discontinuous at every rational value of $\alpha$, except at $\alpha=\beta$ (if $\beta$ is rational). At this value, when $\alpha=\beta$, the constant $C_{\alpha,\alpha}$ above simplifies to $$C_{\alpha,\alpha} = \frac{1}{\sqrt{2\pi\alpha(1-\alpha)}}$$ whether $\alpha$ is rational or not.  Using this, the normal distribution of the $\alpha$-positioned prime factor \eqref{eq:normal} follows as a corollary.  In fact we can improve the error term in that expression to the following.
\begin{corollary} Fix $\alpha \in (0,1)$. We have, uniformly for all real $t$, 
\begin{equation}\frac 1x \#\left\{n \le x: \frac {\log_2 P^{(\alpha)}(n)-\alpha\log_2 x}{\sqrt{\log_2 x}} < t\right\} = \Phi\left(\frac{t}{\sqrt{\alpha(1-\alpha)}}\right) + O_\alpha\left(\frac{\left(\log_3 x\right)^{3/2}}{\sqrt{\log_2 x}}\right).\label{eq:improvednormal}\end{equation}
\end{corollary}
The proof proceeds by partial summation over primes of $\Mpax$, combined with the methods used in the proof of Theorems \ref{thm:RpNormal} and \ref{thm:avglogmid} below.

We can similarly derive as a corollary a description of the distribution of the relative position of a fixed prime $p$ among the prime divisors of integers divisible by $p$.  For an integer $n$ that is divisible by $p$, we write $n= p_1 \cdots p_{\Omega(n)}$, where \[p_1\leq p_2 \leq \cdots \leq p_{k-1} \leq  p=p_k < p_{k+1} \leq \cdots \leq p_{\Omega(n)}\] so that $p$ is the $k$-th smallest prime factor of $n$ (and if $n$ is divisible by $p^2$, we take the largest index corresponding to a factor of $p$). 
For such an $n$ we then denote by $R_p(n)=\frac{k}{\Omega(n)}$ the relative position of $p$ among the prime factors of $n$.   It isn't hard to show that the normal order of $R_p(n)$ is $\beta=\frac{\log_2 p}{\log_2 x}$; we show that it is in fact normally distributed around this value.

\begin{theorem}\label{thm:RpNormal}
Fix $\epsilon > 0$. As $x \rightarrow \infty$, we have
\begin{equation}\label{eq:RpNormal}
\frac1{x/p}\#\left\{n \le x: \, p|n, \frac {R_p(n)-\beta}{(\log_2 x)^{-1/2}} < t\right\} = \Phi\left(\frac{t}{\sqrt{\beta(1-\beta)}}\right) + O_{\epsilon}\left(\frac1{(\log_2 x)^{1/3}}\right),
\end{equation}
uniformly in $\beta \in (\epsilon, 1-\epsilon)$ and all real $t$. %
\end{theorem} 

We conclude with one more application of Theorem \ref{thm:middle}. If we denote by $P_k(n)$ the $k$-th \textit{largest} prime factor of $n$, Dickman \cite{Dickman}, de Bruijn \cite{dB51} and later Knuth and Trabb-Pardo \cite{KTP} investigate the average value of $\log P_k(n)$ and find, for each fixed $k$, that \[\frac 1x \sum_{n\le x} \log P_k(n) = (D_k+o(1))\log x \] where the $D_k$ are constants and in particular $D_1 = 0.624329\ldots$ is the Golumb-Dickman constant.\footnote{In fact, $\frac 1x \sum\limits_{n \le x} \log P_1(n) = D_1\log x + D_1(1-\gamma) + O\left(\exp(-(\log x)^{3/8-\epsilon})\right)$.}

This result can be interpreted as saying that ``on average'' the largest prime divisor of an integer $n$ has just under $5/8$ as many digits as $n$, and for any fixed $k$ the $k$-th largest prime factor of $n$ has, on average, a fixed, positive proportion of the number of digits that $n$ has. Tenenbaum  \cite{ten99}*{Corollary 4} considers this same average for the least prime factor $P^{-}(n)$ and shows (for a complicated but explicit constant $A^-$) that \[\frac 1x \sum_{1<n \le x} \log P^{-}(n) = e^{-\gamma} \log \log x +A^- + O\left(\exp(-(\log x)^{3/8-\epsilon})\right).\] 
We investigate the same problem for the middle (or $\alpha$-positioned) prime factor of $n$.

\begin{theorem} \label{thm:avglogmid} Let $\varphi = \frac{1+\sqrt{5}}{2}$ be the golden ratio, and $\varphi' = \frac{1}{\varphi} = \varphi-1 = \frac{\sqrt{5} -1}{2}=0.6180\ldots$ its reciprocal. The average value of the logarithm of the middle prime factor of the integers up to $x$ satisfies
\begin{equation}\frac 1x \sum_{n \le x} \log P^{\left(\frac 12 \right)}(n) =  A(\log x)^{\varphi'} \left(1+O\left(\frac{(\log_3 x)^{3/2}}{\sqrt{\log_2 x}}\right)\right)\label{eq:golden}\end{equation}  
where \[A \coloneqq \frac{e^{-\gamma}}{\varphi!}\frac{\varphi + 1}{\sqrt{5}}\prod_p\left(1-\frac{1}{p}\right)^{\varphi'}\left(1-\frac{\varphi'}{ p }\right)^{-1} = 1.313314\ldots\] 
and $\varphi! \coloneqq \Gamma(\varphi+1) = \Gamma(\varphi-1) = \Gamma(\varphi')$. 
\end{theorem}
Note that $\varphi$ is one of the solutions to the equation $\Gamma(x+1) = \Gamma(x-1)$.  One can similarly generalize this to other values of $\alpha$.

\begin{remark}
Using Theorem \ref{thm:AlphaModified} one can similarly derive that for any fixed $0<\alpha<1$ we have $\frac 1x \sum\limits_{n \le x} \log P^{\left(\alpha \right)}(n) \sim  A_\alpha(\log x)^{B_\alpha}$, where $B_\alpha = \max\limits_{0<\beta<1}\{\beta + \left(\frac{\beta}{\alpha}\right)^\alpha\left(\frac{1-\beta}{1-\alpha}\right)^{1-\alpha}-1\}$. %
\end{remark}

\subsection*{Notation and conventions:} Most of our notation is standard. We continue to use $P^{+}(n)$ for the largest prime factor of $n$ (with $P^{+}(1)=1$) and we use $P^{-}(n)$ for the smallest prime factor of $n$ (taking $P^{-}(1)=\infty$). We say $n$ is \textsf{$y$-smooth} if $P^{+}(n) \le y$, and we call $n$ \textsf{$y$-rough} when $P^{-}(n) > y$. Given a set of primes $E$, we use $\Omega_E(n)$ to denote the number of primes of $E$ dividing $n$, counted with multiplicity; explicitly, $\Omega_E(n)\coloneqq \sum_{p^k \parallel n,\ p\in E} k$. We also denote by $E(x)$ the sum of reciprocals of the elements of $E$ up to $x$, that is, $E(x)\coloneqq \sum_{\substack{p \le x\\p \in E}} 1/p$.

Implied constants in $\ll$ and $O$-notation may always depend on any parameter declared as ``fixed''. %
In particular, they depend on $\alpha$ and $\epsilon$ unless stated otherwise. %
For rational $\alpha$, the dependence on $\alpha$ will come from the distance of $\alpha$ from $0$ and $1$, and the size of the denominator of $\alpha$. For fixed quantities $\delta$ and $\epsilon$, we shall write $\delta \ll_\epsilon 1$ to mean that $\delta \in (0, 1)$ may be fixed to be sufficiently small in terms of $\epsilon$ (we shall be using variants of this notation with $\delta$ and $\epsilon$ replaced by other fixed parameters). We write $\log_k x$ for the $k$-fold iterate of the  natural logarithm. %
\section{The exact middle prime factor}
Let $\overline{M}_p(x)$ denote the number of integers $n\leq x$ with $\Omega(n)\equiv 1 \pmod{2}$, and whose exact middle prime factor is $p$.  In order to present our argument in as simple a manner as possible we will first prove the following theorem, which may also be of some interest in its own right.

\begin{theorem} \label{thm:main} Let $\epsilon>0$ and suppose $p \to \infty$, $\beta = \frac{\log_2 p}{\log_2 x}$.  Then if either $\beta<\frac 15-\epsilon$ or $\frac 15 +\epsilon < \beta <1-\epsilon$ we have
\begin{equation}\overline{M}_p(x) = \begin{cases}\displaystyle{\left(1+O_\epsilon\left(\sqrt{\frac{\log_3 x}{\log_2 x}}\right)\right)\overline{C}_\beta \displaystyle{\frac{x}{p(\log x)^{1-2\sqrt{\beta(1-\beta)}}\sqrt{\log_2 x}}}} & \text{if  }\frac{1}{5} +\epsilon < \beta <1-\epsilon,\\
\vspace{-2mm}\\
\displaystyle{\left(1+O_\epsilon\left(\sqrt{\frac{\log_3 x}{\log_2 x}} + \frac{(\log_2 p)^{-1/2}}{(\log p)^{\epsilon^2}}\right)\right)\overline{C} \displaystyle{\frac{x}{p(\log x)^{\frac 12 -\frac 32 \beta}}}} &  \text{if  }0 <\beta<\frac{1}{5} - \epsilon, 
\end{cases}\label{eq:exactmideqModif}\end{equation}
where \begin{align*}
\overline{C}_\beta &\coloneqq \frac{\exp\left(\frac{\gamma(1-2\beta)}{\sqrt{\beta(1-\beta)}}\right)}{\Gamma\left(1+\sqrt{\frac{\beta}{1-\beta}}\right)}\frac{\beta^{1/4}}{2\sqrt{\pi}(1-\beta)^{3/4}}\prod_{q \text{ prime}}\left(1-\frac{1}{q}\right)^{\sqrt{\frac{1-\beta}{\beta}}}\left(1-\parbox{1.08cm}{$\frac{\mbox{\tiny $\sqrt{\tfrac{1-\beta}{\beta}}$}}{\mbox{\normalsize $ q$ }}$}\right)^{-1},\\
\overline{C} &\coloneqq \frac{e^{\frac{3\gamma}{2}}}{4\sqrt{\pi}}\prod_{q> 2 \text{ prime}} \left(1+\frac{1}{q(q-2)}\right)=0.507851\ldots.
\end{align*}

\end{theorem}

Analogously to $\overline{M}_p(x)$, we can define $\underline{M}_p(x)$ to be the count of those integers having an even number of prime factors, and whose middle prime factor is $p$.  $\underline{M}_p(x)$ satisfies an expression identical to \eqref{eq:mideqModif} above but with different constants $\underline{C}_\beta$ and $\underline{C}$ in place of $\overline{C}_\beta$ and $\overline{C}$ respectively, namely, $\underline{C}_\beta \coloneqq \overline{C}_\beta\sqrt{\frac{1-\beta}{\beta}}$ and $\underline{C} \coloneqq 2\overline{C} = 1.015703\ldots$. Summing the expressions for $\overline{M}_p(x)$ and $\underline{M}_p(x)$ gives the expression for $M_p(x)$ in Theorem \ref{thm:middle}.

\section{Technical Preparation} 
Before proving our main results, we state several results which will be used in the proofs. We begin with the following consequences of the classical results of Sathe--Selberg and Nicolas concerning the distribution of numbers with a given number of prime factors (see, e.g., Theorems 6.5 and 6.6 on p. 304, and Exercise 217 of \cite{tenenbaumAPNT}).

\begin{lemma}\label{lem:TEN} Fix $\delta \in (0,1)$. For all sufficiently large values of $x$,
\[ \sum_{\substack{n\le x \\ \Omega(n)=k}} 1 \ll \frac{x}{\log{x}} \frac{(\log_2 x)^{k-1}}{(k-1)!} \]
uniformly for positive integers $k \le (2-\delta)\log_2 x$, and 
\[ \sum_{\substack{n\le x \\ \Omega(n)=k}} 1 \ll \frac{x \log{x}}{2^k}\]
uniformly for $k \ge (2+\delta)\log_2 x $. 
\end{lemma} 

The following lemma belongs to the study of the `anatomy of integers', and makes precise the claim that $\Omega_E(n) = \sum_{p^k \parallel n,\ p\in E} k$ is typically of size $\sum_{p\le x,~p\in E} 1/p$, uniformly across all sets of primes $E$. %
Although the statement below is slightly more general than Lemma 3.1 in \cite{MPS}, the same proof goes through; also compare with Theorem 08 on p.\ 5 of \cite{HT88}. 
\begin{lemma}\label{lem:HT} Fix $\epsilon_0 \in (0, 1), C_0>0$. Let $x\ge 3$ and let $E$ be a nonempty set of primes with smallest element $p_E$. With $E(x) = \sum_{p \le x,~p \in E} 1/p$, we have, for $1\le y \le \min\{C_0,(1-\epsilon_0)p_E\}$,
\begin{align*} \sum_{\substack{n \le x \\ \Omega_E(n) \ge y E(x)}} 1 &\ll x\exp(-E(x)\cdot Q(y)), 
\end{align*}
where $Q(y) \coloneqq y \log y - y + 1$ and the implied constant is absolute. When $0 < y \le 1$, the same inequality holds with the $\Omega_E(n)$ condition replaced by $\Omega_E(n) \le y E(x)$. 
\end{lemma}

We denote by $\Phi_k(x,y)$ the number of integers $n\leq x$ whose least prime factor $P^{-}(n) \ge y$ and where $\Omega(n)=k$. The following two results of Alladi provide estimates for $\Phi_k(x, y)$ in different ranges of $y$. 
\begin{theorem}[Alladi \cite{Alladi}, Theorem 7, see also \cite{Bal90}]  \label{thm:Alladi} %
Fix $r>0$ and set $u \coloneqq \frac{\log x}{\log y}$, $\xi \coloneqq \frac{k}{\log u -\gamma}$. Then uniformly for $\exp((\log_2 x)^3) \le y \le \sqrt{x}$ and $1\leq k \le r \log u$, we have\footnote{Note that while Theorem 7 of \cite{Alladi} is stated only for $k \leq (2-\epsilon)\log u$, the concluding remarks of that paper point out how that restriction can be weakened to the version given here.}
$$\Phi_k(x,y) = \frac{xe^{-\gamma\xi}}{\log x \  \Gamma(1+\xi)} \cdot \frac{(\log u)^{k-1}}{(k-1)!}\left(1+O_r\left(\frac{1}{\sqrt{\log u}}\right)\right).$$
  
\end{theorem}

\begin{theorem}[Alladi \cite{Alladi}, Theorem 6] 
\label{thm:Alladi2} Fix $\epsilon \in (0, 2)$ and set $\mu \coloneqq \frac{k-1}{\log_2 x}$. Then uniformly for $3\le y \le \exp\left((\log x)^{2/5}\right)$ and $1\leq k \le (2-\epsilon)\log_2 x$, we have
$$\Phi_k(x,y) = \frac{x g(y,\mu)}{\log x \  \Gamma(1+\mu)} \cdot \frac{(\log_2 x)^{k-1}}{(k-1)!}\left(1+O\left(\frac{k(\log_2 y)^2}{ (\log_2 x)^2}\right)\right),$$
where \[g(y,\mu) \coloneqq \prod_{p<y} \left(1-\frac{1}{p}\right)^\mu \prod_{p\geq y}\left(1-\frac{1}{p}\right)^\mu\left(1-\frac{\mu}{p}\right)^{-1}.\]

\end{theorem}

The truncated sums of the exponential series (and twists thereof) will play a starring role in our arguments. The following technical result provides the groundwork for estimating such sums. 
\begin{theorem}[Norton \cite{Norton}, Lemmas 4.5 and 4.7] \label{thm:Norton}
Let $v,\theta$ be positive real numbers. Then
\begin{equation} \label{eq:Norton45}
\sum_{k \le (1-\theta)v} \frac{v^k}{k!} < \frac{1}{\theta \sqrt{ v(1-\theta)}} \exp\left((R(-\theta)+1) v\right)\end{equation} 
and \begin{equation} \label{eq:Norton47}
 \sum_{k \ge (1+\theta)v} \frac{v^k}{k!} < \frac{\sqrt{1+\theta}}{\theta \sqrt{2\pi v}} \exp\left((R(\theta)+1) v\right)
\end{equation} 
where $R(\theta) = \theta-(\theta+1)\log(\theta+1)$.
\end{theorem}

We will be making more frequent use of the following consequence of Norton's estimates.
\begin{lemma}\label{lem:NortonAppGen} Suppose $W>0$ is sufficiently large and that $E = o(W^{2/3})$ as $W \rightarrow \infty$. Then
$$\sum_{\substack{0 \leq k \le W-E}} \frac{W^k}{k!} \ll \frac{W^{1/2}}{E} \exp\left(W-\frac{E^2}{2W} \right) \ \ \ \ \text{and} 
\ \ \sum_{\substack{k \ge W+E}} \frac{W^k}{k!} \ll \frac{W^{1/2}}{E} \exp\left(W-\frac{E^2}{2W} \right).$$
\end{lemma}
\begin{proof}
For any parameter $\theta = o(W^{-1/3})$, we have 
\begin{equation*}
\begin{split}
R(-\theta)+1 &= (1-\theta)(1 - \log(1-\theta)) = (1-\theta)\left(1 + \theta + \frac{\theta^2}2 + O(\theta^3)\right) = 1 - \frac{\theta^2}2 + O(\theta^3).
\end{split}
\end{equation*}
Consequently taking $\theta = E/W = o(W^{-1/3})$, we deduce from \eqref{eq:Norton45} and the above estimate that
$$\sum_{\substack{0 \leq k \le W-E}} \frac{W^k}{k!} = \sum_{k \le (1-\theta)W} \frac{W^k}{k!} \ll \frac{e^W}{\theta \sqrt W} \exp\left(-\frac12 \theta^2 W\right) \ll \frac{W^{1/2}}{E} \exp\left(W-\frac{E^2}{2W} \right),$$
establishing the first of the claimed estimates. The proof of the second is analogous. \qedhere
\end{proof}

We conclude this section with an estimate on certain `twisted' versions of the truncated sums of the exponential series, which arise in the arguments of Theorem \ref{thm:AlphaModified}.
\begin{lemma} \label{lem:beattyModified} 
Fix $\epsilon, \alpha \in (0, 1)$ with $\alpha$ rational. Then we have, uniformly in $\beta \in [\betalpPl, 1-\epsilon]$ and in $V, E \rightarrow \infty$ satisfying %
$V^{1/2} \ll E$ and $E=o(V^{2/3})$, 
\begin{equation}\label{eq:LargeBetaMTRat}
\sum_{V-E \le k \le V+E} \frac{V^k}{k!} \chiexpfr = \rho_{\chi, \alpha} e^V \left\{1+O\left(e^{-c_\alpha V} + \frac{V^{1/2}}E \exp\left(-\frac{E^2}{2V}\right)\right)\right\}
\end{equation}
where $c_\alpha>0$ is a constant depending only on $\alpha$. With the same restrictions on $\alpha, V$ and $E$, 
\begin{equation}\label{eq:SmallBetaMTRat}
\sum_{\frac{V-E}{1-\alpha} \le k \le \frac{V+E}{1-\alpha}}  \frac{V^{\lfloor{(1-\alpha)k\rfloor}}}{\lfloor{(1-\alpha)k\rfloor}!} \nuexpfr = \frac{\rho_{\nu, \alpha} e^V}{1-\alpha} \left\{1+O\left(\sqrt{\frac EV} + \frac{V^{1/2}}E \exp\left(-\frac{E^2}{2V}\right)\right)\right\}.
\end{equation}
The implied constants in the two formulae above depend at most on the distance of $\alpha$ from $0$ and $1$, and the size of the denominator of $\alpha$. 

The above asymptotic formulae also hold true (without any effective error term) %
for a fixed irrational $\alpha \in (0, 1)$, and with a multiplicative error of
\begin{equation}\label{eq:IrrAlphaErr}
1+O\left(\frac{E^{3/2}}V + \frac{ \log V (\log_2 V)^{1+\epsilon}}{E^{1/2}} + \frac{V^{1/2}}E \exp\left(-\frac{E^2}{2V}\right)\right) \text{  for almost all }\alpha \in (0, 1). 
\end{equation}
Finally, given $\epsilon, V, E$ as above and a parameter $\err = o(1)$, we have 
\begin{equation}\label{eq:beattynearbeta}
\sum_{V-E \le k \le V+E} \frac{V^k}{k!} \chiexpfr = e^V \left\{1+O\left(\err + \frac{V^{1/2}}E \exp\left(-\frac{E^2}{2V}\right)\right)\right\},    
\end{equation}
uniformly in $\beta \in (\epsilon, 1-\epsilon)$ and in $\alpha \in (\beta-\err, \beta+\err)$.
\end{lemma} 

\begin{proof}
We first consider the case when $\alpha \in (0, 1)$ is rational. We start by writing $1-\alpha = a/b$ for some coprime positive integers $a, b$, so that $b>1$. Then the sum on the left hand side of \eqref{eq:LargeBetaMTRat} is equal to 
\begin{equation}\label{eq:ResClassSplit}
\sum_{V-E \le k \le V+E} \frac{V^k}{k!} \chi^{\{ak/b\}} = \sum_{r \bmod b} \chi^{\{ar/b\}} \sum_{\substack{V-E \le k \le V+E\\k \equiv r \pmod b}} \frac{V^k}{k!}.    
\end{equation}
By the orthogonality of additive characters, the inner sum on $k$ is (writing $e(x) \coloneqq e^{2\pi i x}$) 
\begin{equation*}
\begin{split}
\frac1b \sum_{\ell \bmod b} & e\left(-\frac{r\ell}b\right) \sum_{V-E \le k \le V+E} \frac{(Ve^{2\pi i \ell/b})^k}{k!} \\&= \frac1b \sum_{\ell \bmod b} e\left(-\frac{r\ell}b\right) \exp(Ve^{2\pi i \ell/b}) + O\left(e^V\frac{V^{1/2}}E \exp\left(-\frac{E^2}{2V}\right)\right). 
\end{split}
\end{equation*}
Plugging this back into \eqref{eq:ResClassSplit} and interchanging sums yields
\begin{equation*}
\begin{split}
&\sum_{V-E \le k \le V+E} \frac{V^k}{k!} \chiexpfr \\&= \frac{e^V}b \sum_{r \bmod b} \chi^{\{ar/b\}} + \frac1b \sum_{\ell=1}^{b-1} \exp(Ve^{2\pi i \ell/b}) \sum_{r \bmod b} \chi^{\{ar/b\}} e\left(-\frac{r\ell}b\right) + O\left(e^V\frac{V^{1/2}}E \exp\left(-\frac{E^2}{2V}\right)\right)\\
&= \rho_{\chi, \alpha} e^V + \frac1b \sum_{\ell=1}^{b-1} \exp(Ve^{2\pi i \ell/b}) \sum_{j=0}^{b-1} \chi^{j/b} e\left(-\frac{\overline a j \ell}b\right) + O\left(e^V\frac{V^{1/2}}E \exp\left(-\frac{E^2}{2V}\right)\right)
\end{split}
\end{equation*}
where $\overline a \in \mathbb Z$ denotes a multiplicative inverse of $a$ mod $b$, and we have noted that as $r$ runs over the different residues mod $b$, so does $ar$. The inner sum in the last display is $O(1)$, %
so that the sum over $\ell$ above is $\ll e^{(1-c_\alpha)V}$ with $c_\alpha \coloneqq 1-\cos(2\pi/b)>0$. This completes the proof of \eqref{eq:LargeBetaMTRat}.

In order to show \eqref{eq:SmallBetaMTRat}, we start by setting $m = \lfloor (1-\alpha) k \rfloor$, so that $m \in (V-E-1, V+E]$ and $m/(1-\alpha) \le k < (m+1)/(1-\alpha)$. This last condition automatically implies that $k \in [(V-E)/(1-\alpha), (V+E)/(1-\alpha)]$ for all $m \in [V-E, V+E-1]$. The remaining $m$ are of the form $V+\theta E + O(1)$ for some $\theta \in \{\pm 1\}$, so that for such $m$,
$$h(m) \coloneqq m\log V - m\log m + m -\frac12 \log m = V - \frac12 \log V - \frac{E^2}{2V} + O(1),$$
and by Stirling's formula, 
\begin{equation}\label{eq:Vmm!Extrm}
\frac{V^m}{m!} = \exp\left(h(m)-\frac12 \log(2\pi)\right) \left(1+O\left(\frac1m\right)\right) \ll \frac{e^V}{V^{1/2}} \exp\left(-\frac{E^2}{2V}\right),
\end{equation}
which is negligible compared to the error term in \eqref{eq:SmallBetaMTRat}. Hence, up to a negligible error, the sum in \eqref{eq:SmallBetaMTRat} is equal to
\begin{equation}\label{eq:SmallBetaSumAliter}
\sum_{V-E \le m \le V+E} \frac{V^m}{m!} \sum_{\frac m{1-\alpha} \le k < \frac {m+1}{1-\alpha}} \nuexpfr    .
\end{equation}
Now, for any positive $A<B$, we see that 
\begin{equation}\label{eq:FracIntervRat}
\sum_{A \le k < B} \nuexpfr = \sum_{r \bmod b} \nu^{\{ar/b\}} \sum_{\substack{A \le k < B\\k \equiv r \pmod b}} 1 = (B-A) \rho_{\nu, \alpha} + O(1).    
\end{equation}
Defining $L \coloneqq \sqrt{V/E}$, we partition the interval $[V-E, V+E]$ into $\lfloor 2E/L \rfloor$ equal length subintervals $I$, so that each subinterval has length $2E (2E/L + O(1))^{-1} = L + O(V/E^2) = L + O(1)$. %
For any subinterval $I$ of length $L+O(1)$ and any two integers $m_1, m_2 \in I$, Lagrange's Mean Value Theorem implies that there exists $m' \in [m_1, m_2]$ satisfying 
$$h(m_1) - h(m_2) = (m_1 - m_2) \frac{\partial h}{\partial m}\Big\vert_{m=m'} \ll L\left\{-\log\left(1-\frac EV\right) + \frac1V\right\} \ll \frac{EL}V,$$
whereby another application of Stirling's formula reveals that
\begin{equation}\label{eq:gm_iRelation}
\frac{V^{m_1}}{m_1!} = \frac{V^{m_2}}{m_2!}\left(1+O\left(\frac{EL}V\right)\right) =  \frac{V^{m_2}}{m_2!}\left(1+O\left(\sqrt{\frac EV}\right)\right).
\end{equation}
As such, fixing some integer $m_I \in I$ and letting $L_I$ and $R_I$ denote the least and largest integers in $I$ respectively, we see that the contribution of all $m \in I$ to the sum \eqref{eq:SmallBetaSumAliter} is
\begin{equation*}
\begin{split}
\sum_{m \in I} \frac{V^m}{m!} \sum_{\frac m{1-\alpha} \le k < \frac {m+1}{1-\alpha}} \nuexpfr &= \frac{V^{m_I}}{m_I!} \left(1+O\left(\sqrt{\frac EV}\right)\right) \sum_{\frac{L_I}{1-\alpha} \le k < \frac {R_I+1}{1-\alpha}} \nuexpfr\\
&\stackrel{\eqref{eq:FracIntervRat}}= \frac{\rho_{\nu, \alpha}}{1-\alpha} \left(1+O\left(\sqrt{\frac EV}\right)\right) L\frac{V^{m_I}}{m_I!}\\
&\stackrel{\eqref{eq:gm_iRelation}}= \frac{\rho_{\nu, \alpha}}{1-\alpha} \left(1+O\left(\sqrt{\frac EV}\right)\right) \sum_{m \in I} \frac{V^m}{m!}.
\end{split}
\end{equation*}
Summing this over all $I$, and using Lemma \ref{lem:NortonAppGen} to extend the sum on $m$, we obtain \eqref{eq:SmallBetaMTRat}. 

Now, given a parameter $\err=o(1)$, we see that $\chi = (1-\alpha)\beta/(1-\beta)\alpha = 1 + O(\err)$ uniformly for $\beta \in (\epsilon, 1-\epsilon)$ and $\alpha \in (\beta-\err, \beta+\err)$, so that $\chiexpfr = \exp(\{(1-\alpha)k\} \log \chi) = 1+O(\err)$, and Lemma \ref{lem:NortonAppGen} completes the proof of \eqref{eq:beattynearbeta}.

Finally, in order to establish the assertions corresponding to both \eqref{eq:LargeBetaMTRat} and \eqref{eq:SmallBetaMTRat} for irrational $\alpha \in (\epsilon, 1-\epsilon)$, we carry out the above argument with $L \coloneqq E^{1/2}$. By \eqref{eq:gm_iRelation}, the sums in \eqref{eq:LargeBetaMTRat} and \eqref{eq:SmallBetaSumAliter} are respectively equal to 
\begin{equation*}
\sum_I \frac{V^{k_I}}{k_I!} \!\left(1+O\!\left(\frac{E^{3/2}}V\right)\!\right) \sum_{k \in I} \chiexpfr \text{ \ and \ } \sum_I \frac{V^{m_I}}{m_I!} \!\left(1+O\!\left(\frac{E^{3/2}}V\right)\!\right) \!\sum_{\frac{L_I}{1-\alpha} \le k < \frac {R_I+1}{1-\alpha}} \nuexpfr,
\end{equation*}
where (as before) the outer sums are over the $\lfloor 2E/L \rfloor$ equal length subintervals $I$ partitioning $[V-E, V+E]$, and $k_I$ and $m_I$ are some integers chosen from $I$.

As such, it only remains to estimate the sums $\sum_{k \in J} c^{\{(1-\alpha)k\}}$ uniformly over intervals $J \subset [V-E, V+E]$ with length $|J| \rightarrow \infty$, where $c \in \{\chi, \nu\}$. Since the sequence $\{(1-\alpha)k\}_{k=1}^\infty$ is uniformly distributed mod $1$, this sum is $\sim \rho_{c, \alpha} |J|$. In fact, by Koksma's inequality (Theorem 5.4 in \cite{harman}), we see that 
$$\left|\frac1{|J|+O(1)} \sum_{k \in J} c^{\{(1-\alpha)k\}} - \int_0^1 c^t \, \mathrm dt\right| \le \text{Var}(t \mapsto c^t) \cdot \text{Disc}(\{(1-\alpha)k : k \in J\}),$$
where $\text{Var}(t \mapsto c^t)$ denotes the total variation of the function $t \mapsto c^t$ on $[0, 1)$ and $\text{Disc}(\{(1-\alpha)k : k \in J\})$ denotes the discrepancy of the sequence $\{(1-\alpha)k : k \in J\}$. By Khintchine's bound (Theorem 5.15 in \cite{harman}), the above discrepancy is $\ll \log L (\log_2 L)^{1+\epsilon}/L$ for almost all $\alpha \in (0, 1)$, while $\text{Var}(t \mapsto c^t) = |c-1| \ll 1$. Carrying out the above simplifications in reverse completes the proof of the lemma. 
\end{proof}
\section{Mertens' Theorem dissected} %
We will need the following result of Lichtman estimating the sum of reciprocals of smooth numbers with a given number of prime factors. 
\begin{theorem}[Lichtman \cite{Lichtman}, Theorem 4.1] \label{thm:lichtman}
Fix $\epsilon > 0$, and set $r \coloneqq \frac{k}{\log_2 y}$ and $$\eta(z) \coloneqq e^{\gamma z} \prod_p \left(1-\frac{1}p\right)^z\left(1-\frac{z}{p}\right)^{-1}.$$  As $y \to \infty$, we have, uniformly for $k\leq (2-\epsilon)\log_2 y$,
$$\sum_{\substack{P^+(A) \le y\\\Omega(A)=k}} \frac{1}{A} = \eta(r) \frac{(\log_2 y)^k}{k!}\left(1+ O_\epsilon\left(\frac{k}{(\log_2 y)^2}\right)\right).$$   
\end{theorem} %

It will also be helpful to have some upper bound for the above sums %
that is valid for all values of $k$. Such a result can be obtained by an application of Rankin's method. 
\begin{lemma} \label{lem:coarsebd} We have uniformly for $y \ge 3$ and integers $J\geq 1$,
\begin{equation}\sum_{\substack{P^+(A) \le y\\ \Omega(A) = J}} \frac{1}{A} \ll \frac{J}{2^J}\log^2 y. \label{eq:coarsebd}
\end{equation}
\end{lemma}

\begin{proof}

For any $0<z<2$ we have 
\begin{align*}
\sum_{\substack{P^+(A)\le y\\ \Omega(A) = J}} \frac 1A < z^{-J} \sum_{\substack{P^+(A)\leq y}} \frac{z^{\Omega(A)}}{A} &= z^{-J} \left(1-\frac{z}{2}\right)^{-1} \prod_{3 \le p\leq y} \left(1-\frac{z}{p}\right)^{-1}\\ &\ll \frac{z^{-J}}{2-z} \exp\left(z\sum_{3 \le p \le y} \frac1p \right) \ll \frac{z^{-J}}{2-z} (\log y)^z.
\end{align*}
Taking $z = 2-1/J$ and noting that $(1-1/2J)^{-J} \asymp 1$, we obtain the desired estimate. \qedhere 
\end{proof} 

We will also need a version of Lichtman's theorem (Theorem \ref{thm:lichtman}) for ``large'' values of $k$, namely those where $k>(2+\epsilon)\log_2 y$. In fact, we show that \eqref{eq:coarsebd} above essentially gives the correct order of magnitude (up to the factor of $k$ in the numerator) for such $k$.  This result can be viewed as an extension of Lichtman's result on dissecting Mertens' theorem for very large values of $k$.  We define \[ \eta_o(z) \coloneqq e^{\gamma z} 2^{-z} \prod_{p>2}\left(1-\frac{1}p\right)^z\left(1-\frac{z}{p}\right)^{-1}\]
so that $\eta_0(z) = (1-z/2)\eta(z)$ for all $z \ne 2$.
\begin{theorem}\label{thm:ImprExtnLicht1} 
Fix $\epsilon \in (0, 1/2)$ and $A>1$. We have uniformly for $y \gg 1$ and $(2+\sqrt{5\epsilon}) \log_2 y \le k \le (\log y)^{1/2-\epsilon}$,
\begin{equation}\label{eq:ImpExtLich1}
\sum_{\substack{P^+(A) \le y\\ \Omega(A)=k}} \frac1A = \eta_o(2) \frac{\log^2 y}{2^k} \left(1 + O\left(\frac1{(\log y)^\epsilon \sqrt{\log_2 y}}\right)\right).
\end{equation}

\end{theorem}
Note that \begin{equation*}
\eta_o(2) = \frac{e^{2\gamma}}{4}\prod_{p> 2} \left(1+\frac{1}{p(p-2)}\right)=1.201303\ldots. 
\end{equation*}
\begin{proof}
In what follows, let $\epsilon_1 \coloneqq \sqrt{5\epsilon}$. We adapt the proof of \cite{tenenbaumAPNT}*{Theorem II.6.6}. 
The sum on the left hand side of \eqref{eq:ImpExtLich1} is the coefficient of $z^k$ in the function \[\sum\limits_{n:\,P^+(n) \le y} \frac{z^{\Omega(n)}}n = \prod_{p \le y} \left(1-\frac zp\right)^{-1},\] which is holomorphic on the disk $|z|<2-\epsilon/2$. As such, the sum in \eqref{eq:ImpExtLich1} equals 
\begin{equation*}%
\frac1{2\pi i} \oint\limits_{|z|=2-\epsilon} \prod_{p \le y} \left(1-\frac zp\right)^{-1} \, \frac{\mathrm{d}z}{z^{k+1}}.   
\end{equation*}
By the Prime Number Theorem (with the usual de la Vall\'ee Poussin error term), 
$$\prod_{p \le y} \left(1-\frac1p\right)^z = \frac{e^{-\gamma z}}{(\log y)^z}\left(1+O(\exp(-K\sqrt{\log y}))\right)$$
for some absolute constant $K>0$. Moreover, for $|z| \le 2-\epsilon$, we have 
\begin{equation*}
\prod_{p>y} \left(1-\frac1p\right)^z \left(1-\frac zp\right)^{-1} = \exp\left(\sum_{p>y} \left\{z\log \left(1-\frac1p\right) - \log\left(1-\frac zp\right)\right\}\right) = 1+O\left(\frac1{y\log y}\right).
\end{equation*}
Consequently, 
\begin{equation}\label{eq:CauchyIntFor2}
\sum_{\substack{P^+(A) \le y\\ \Omega(A)=k}} \frac1A = I + O\left(\exp(-K \sqrt{\log y}) \int_{|z| = 2-\epsilon} \frac{(\log y)^{\Re z}}{|z|^{k+1}} \, |dz|\right),
\end{equation}
where 
\begin{equation}\label{eq:I0}
I \coloneqq \frac1{2\pi i} \oint\limits_{|z|=2-\epsilon} \frac{\eta(z) (\log y)^z}{z^{k+1}} \, \mathrm{d}z = \eta_o(2) \frac{\log^2 y}{2^k} - \frac1{2\pi i} \oint\limits_{|z|=2+\epsilon_1} \frac{\eta(z) (\log y)^z}{z^{k+1}} \, \mathrm{d}z  
\end{equation}
since the residue of the function $\eta(z) (\log y)^z/z^{k+1}$ at the simple pole $z=2$ is precisely $-\eta_o(2)\log^2 y/2^k$. The error term in \eqref{eq:CauchyIntFor2} is $\ll (\log y)^{2-\epsilon} \exp(-K \sqrt{\log y})$, which is negligible compared to the error term in \eqref{eq:ImpExtLich1} since $k \le (\log y)^{1/2-\epsilon}$. Moreover, the last integral in \eqref{eq:I0} is 
\begin{equation*}
\begin{split}
\ll \oint\limits_{|z|=2+\epsilon_1} \frac{(\log y)^{\Re z}}{|z|^{k+1}} \, |\mathrm{d}z| \ll \frac1{(2+\epsilon_1)^k} \int_0^{2\pi} \exp((2+\epsilon_1) \log_2 y \cos\theta) \, \mathrm{d}\theta \ll \frac{(\log y)^{2+\epsilon_1}}{(2+\epsilon_1)^k \sqrt{\log_2 y}}, 
\end{split}
\end{equation*}
where we have used the fact that $\int_0^{2\pi} e^{\lambda \cos\theta} \, \mathrm{d}\theta \ll e^\lambda/\sqrt{\lambda}$ for all $\lambda>1$ (see \cite{tenenbaumAPNT}*{p. 302}). Finally, since $k \ge (2+\epsilon_1) \log_2 y$, the last expression in the above display is 
\begin{equation*}
\begin{split}
 = \frac{(\log y)^{2+\epsilon_1}}{2^k \sqrt{\log_2 y}} \left(\frac2{2+\epsilon_1}\right)^k &\le \frac{(\log y)^{2+\epsilon_1}}{2^k \sqrt{\log_2 y}} \left(\frac2{2+\epsilon_1}\right)^{(2+\epsilon_1) \log_2 y}\\ &= \frac{\log^2 y}{2^k \sqrt{\log_2 y}} (\log y)^{-\tau(\epsilon_1)} \ll \frac{\log^2 y}{2^k} \frac1{(\log y)^\epsilon \sqrt{\log_2 y}},  
\end{split}
\end{equation*} 
where we have noted that for $\epsilon \in (0, 1/2)$, we have $\epsilon_1 = \sqrt{5\epsilon} \in (0, 1.6)$, so that
$$\tau(\epsilon_1) \coloneqq 2\left\{\left(1+\frac{\epsilon_1}2\right) \log\left(1+\frac{\epsilon_1}2\right) - \frac{\epsilon_1}2\right\} \ge \frac{\epsilon_1^2}5 = \epsilon.$$
Collecting estimates completes the proof of the theorem. 
\end{proof}
\begin{remark} 
Although this shall not be essential for us, it is worth noting that the range of $k$ in Theorem \ref{thm:ImprExtnLicht1} can be extended to $(2+\epsilon_y) \log_2 y \le k \le (\log y)^{1/2-\epsilon}$ for any positive parameter $\epsilon_y < 1.6$ depending on $y$: In fact, we can show that in this range of $k$,
\begin{equation*}
\sum_{\substack{P^+(A) \le y\\ \Omega(A)=k}} \frac1A = \eta_o(2) \frac{\log^2 y}{2^k} \left\{1 + O\left(\frac1{\epsilon_y\sqrt{\log_2 y}}\exp\left(-\frac{\epsilon_y^2}5 \log_2 y\right)\right)\right\}, 
\end{equation*}
where the $\epsilon_y^2/5$ may be replaced by $\epsilon_y^2/4$ if $\epsilon_y \ll (\log_2 y)^{-1/3}$ as $y \rightarrow \infty$.

In order to show this, we carry out the above argument until \eqref{eq:I0} with $\epsilon$ replaced by $\epsilon_y$. In order to bound the corresponding analogue of the last integral in \eqref{eq:I0}, we note that since $\eta(z)$ has a simple pole at $z=2$, we have $\eta(z) \ll 1/|z-2| \ll 1/||z|-2| = 1/\epsilon_y$ on the circle $|z| = 2+\epsilon_y$. The above calculations now show that this integral is 
$$\ll \frac{\log^2 y}{2^k} \cdot \frac{\exp(-\tau(\epsilon_y) \log_2 y)}{\epsilon_y \sqrt{\log_2 y}} \ll \frac{\log^2 y}{2^k} \frac1{\epsilon_y\sqrt{\log_2 y}}\exp\left(-\frac{\epsilon_y^2}5 \log_2 y\right),$$
if $\epsilon_y \in (0, 1.6)$. If $\epsilon_y \ll (\log_2 y)^{-1/3}$ as $y \rightarrow \infty$, then we may note that $\tau(\epsilon_y) = (2+\epsilon_y) \log(1+\epsilon_y/2) - \epsilon_y = \epsilon_y^2/4 + O(\epsilon_y^3) = \epsilon_y^2/4 + O(1/\log_2 y)$, allowing us to replace $\epsilon_y^2/5$ by $\epsilon_y^2/4$ in the last bound in the above display. \hfill \qedsymbol
\end{remark}

\section{Proof of Theorem \ref{thm:main}}

We assume throughout that $\beta<1-\epsilon$. We start by %
letting $K \coloneqq 1.02/\log 2 \approx 1.4715$, and note that by the second assertion in Lemma \ref{lem:TEN}, %
the number of $n \le x$ divisible by $p$ which have more than $2K \log_2 x$ prime divisors is $\ll x/p(\log x)^{1.04}$, 
which is negligible for our purposes. Hence, it remains to show that the asymptotic formulae claimed for $\overline M_p(x)$ hold true for the number $N_p (x)$ of positive integers $n \le x$ having $\Omega(n) \equiv 1 \pmod 2$, $\Omega(n) \le 2K \log_2 x$ and $p_{(\Omega(n)+1)/2}(n) = p$.

Any positive integer $n>1$ counted in $N_p(x)$ can be uniquely written in the form $n=ApB$ for some positive integers $A \le x/p$ and $B \le x/Ap$ with $P^+(A) \le p \le P^-(B)$ and with $\Omega(A) = \Omega(B) = k$, where $\Omega(n) = 2k+1 \ge 1$. As such, \begin{equation}\label{NpFirstExpr}
N_p(x) = 1 + \sum_{k \le K \log_2 x} \sum_{\substack{A \le x/p\\ P^+(A) \le p \\ \Omega(A) = k}} \sum_{\substack{B \le x/Ap \\ P^-(B) \ge p \\ \Omega(B) = k}} 1 = \sum_{0 \le k \le K \log_2 x} \sum_{\substack{A \le x/p\\ P^+(A) \le p \\ \Omega(A) = k}} \Phi_k \left(\frac x{Ap}, p\right).
\end{equation}
We first consider the case $p> \exp((\log_2 x)^3)$, or equivalently $\beta> 3 \log_3 x/\log_2 x$. Since $Ap^3 \le p^{k+3} \le \exp(3K \log_2 x (\log x)^{1-\epsilon}) < x^{1/2}$, we have $p \le \sqrt{x/Ap}$, and Theorem \ref{thm:Alladi} yields 
$$N_p(x) = \sum_{k \le K \log_2 x} \sum_{\substack{A \le x/p\\ P^+(A) \le p \\ \Omega(A) = k}} \frac{x/Ap}{\log(x/Ap)} \frac{e^{-\gamma \xi_A} (\log u_A)^{k-1}}{\Gamma(1+\xi_A) (k-1)!} \left(1 + O\left(\frac1{\sqrt{\log u_A}}\right)\right)$$
where $u_A \coloneqq \frac{\log(x/Ap)}{\log p}$ and $\xi_A \coloneqq \frac k{\log u_A - \gamma}$. Now since $\log(Ap) \le (k+1) \log p \ll \log_2 x \log p$, we have $\log\left(\frac x{Ap}\right) = \log x \left(1 + O\left(\frac{\log_2 x \log p}{\log x}\right)\right)$. Consequently, recalling that $u = \frac{\log x}{\log p}$, we obtain
$$u_A = u \left(1 + O\left(\frac{\log_2 x \log p}{\log x}\right)\right) \implies (\log u_A)^{k-1} = (\log u)^{k-1} \left(1 + O\left(\frac{\log_2 x \log p}{\log x}\right)\right),$$
where the implication above uses the fact that $\log u = (1-\beta) \log_2 x > \epsilon \log_2 x$. Likewise, $\xi_A = \xi\left(1 + O\left(\frac{\log p}{\log x}\right)\right)$ where $\xi= \frac k{\log u-\gamma}$, so that $e^{-\gamma \xi_A} = e^{-\gamma \xi} \left(1 + O\left(\frac{\log p}{\log x}\right)\right)$, and by Lagrange's Mean Value Theorem,
$$\Gamma(1+\xi_A) = \Gamma(1+\xi) + O\left(\frac{\log p}{\log x}\right) = \Gamma(1+\xi) \left(1 + O\left(\frac{\log p}{\log x}\right)\right).$$
Collecting estimates, we now obtain
\begin{equation}\label{eq:Npexpr1}
N_p(x) = \frac x{p \log x} \left(1 + O\left(\frac1{\sqrt{\log_2 x}}\right)\right) \sum_{k \le K \log_2 x} \frac{e^{-\gamma \xi}}{\Gamma(1+\xi)} \frac{(\log u)^{k-1}}{(k-1)!} \sum_{\substack{P^+(A) \le p \\ \Omega(A) = k}} \frac 1A,
\end{equation} 
where we have recalled that for any $k \le K \log_2 x$, we have $p^{k+1} \le x$ for all sufficiently large $x$, which implies that any $A$ counted in the inner sum above is automatically $\le x/p$. 

We now write $N_p(x) = N_{p, 1}(x) + N_{p, 2}(x)$, where $N_{p, 1}(x)$ is defined by restricting the outer sum in \eqref{eq:Npexpr1} to $k \le (2-\delta) \log_2 p = (2-\delta) \beta \log_2 x$ for some $\delta \in (0, 1)$, to be chosen small enough in terms of $\epsilon$ (all our statements henceforth will be valid for all $\delta \ll_\epsilon 1$). 

\medskip
\textit{Estimation of $N_{p, 1}(x)$}: In order to estimate $N_{p, 1}(x)$, we invoke Theorem \ref{thm:lichtman} to estimate the inner sum on $A$. This shows that $N_{p, 1}(x)$ is  
\begin{equation*}
\begin{split}
&\frac x{p \log x} \left(1 + O\left(\frac1{\sqrt{\log_2 x}} + \frac 1{\log_2 p}\right)\right) \sum_{k \le (2-\delta) \log_2 p} \eta\left(\frac k{\log_2 p}\right) \frac{e^{-\gamma \xi}}{\Gamma(1+\xi)} \frac{(\log u)^{k-1}}{(k-1)!} \frac{(\log_2 p)^k}{k!}\\
&= \frac x{p \log x \log u} \left(1 + O\left(\frac1{\sqrt{\log_2 x}} + \frac 1{\log_2 p}\right)\right) \sum_{k \le (2-\delta) \log_2 p} \eta\left(\frac k{\log_2 p}\right) \frac{ke^{-\gamma \xi}}{\Gamma(1+\xi)} \binom{2k}k \frac{w^{2k}}{(2k)!}
\end{split}
\end{equation*}
where we have defined 
\[w \coloneqq \sqrt{\log u \log_2 p} = \sqrt{\beta(1-\beta)} \log_2 x,\]
the geometric mean of $\log u $ and $\log_2 p$. We will find that the values of $k$ which are most significant in the sum above are those with $k \approx w$.

By Stirling's estimate, we see that $\binom{2k}k = \frac{2^{2k}}{\sqrt{\pi k}} \left(1+O\left(\frac 1k \right)\right)$; hence
\begin{multline}\label{eq:Np1Expr1}
N_{p, 1}(x) = \frac x{\pi^{1/2} p \log x \log u} \left(1 + O\left(\frac1{\sqrt{\log_2 x}} + \frac 1{\log_2 p}\right)\right)\\ \sum_{k \le (2-\delta) \log_2 p} \eta\left(\frac k{\log_2 p}\right) \frac{k^{1/2} e^{-\gamma \xi}}{\Gamma(1+\xi)} \frac{(2w)^{2k}}{(2k)!} \left(1+O\left(\frac1k\right)\right).
\end{multline}

Now invoking Lemma \ref{lem:NortonAppGen} with $W=2w$ and $E=6\sqrt{w\log w} \eqqcolon 2w'$, we see that
\begin{align}
\sum_{k \le \min\{w-w', (2-\delta) \log_2 p\}} &\eta\left(\frac k{\log_2 p}\right) \frac{k^{1/2} e^{-\gamma \xi}}{\Gamma(1+\xi)} \frac{(2w)^{2k}}{(2k)!} \nonumber \\
&\ll w^{1/2} \sum_{m \le 2w-2w'} \frac{(2w)^m}{m!}\nonumber\\
&\ll \frac{e^{2w}}{w^8 \sqrt{\log w}} \ll \frac{(\log x)^{2 \sqrt{\beta(1-\beta)}}}{(\log_2 x)^4 (\log_3 x)^4}
\label{eq:NortonApp1}
\end{align}
where we have noted that $w = \sqrt{\beta(1-\beta)} \log_2 x \gg \beta^{1/2} \log_2 x \gg \sqrt{\log_2 x \log_3 x}$ since $3 \log_3 x/\log_2 x < \beta < 1-\epsilon$. This shows that the total contribution from $k \le \min\{w-w', (2-\delta) \log_2 p\}$ to the right hand side of \eqref{eq:Np1Expr1} is 
$$\ll \frac 1{(\log_2 x \log_3 x)^4} \frac x{p (\log x)^{1-2\sqrt{\beta(1-\beta)}} \log_2 x}$$%
which is negligible in comparison to our error terms in either of the two cases $\beta<1/5-\epsilon$ or $\beta \in (1/5+\epsilon, 1-\epsilon)$. (For $\beta<1/5-\epsilon$, we make use of the easy fact that $1-2\sqrt{\beta(1-\beta)} > 1/2-3\beta/2$ for all $\beta \in (0, 1)$.) In particular, for $\beta<1/5-\epsilon$, we have $(2-\delta) \log_2 p < w-w'$, hence the above argument shows that the count $N_{p, 1}(x)$ itself is absorbed in the claimed error term.

Now if $\beta \in (1/5+\epsilon, 1-\epsilon)$, then the total contribution of $k \in (w+w', (2-\delta) \log_2 p]$ to the right hand side of \eqref{eq:Np1Expr1} is, by another application of Lemma \ref{lem:NortonAppGen},
\begin{equation*}
\begin{split}
&\ll \frac x{p \log x \log u} \sum_{w+w' < k \le (2-\delta) \log_2 p} \eta\left(\frac k{\log_2 p}\right) \frac{k^{1/2} e^{-\gamma \xi}}{\Gamma(1+\xi)} \frac{(2w)^{2k}}{(2k)!}\\
&\ll \frac{x \sqrt{\log_2 p}}{p \log x \log u} \sum_{m > 2w+2w'} \frac{(2w)^m}{m!} \ll \frac1{(\log_2 x)^9\sqrt{\log_3 x}} \frac x{p (\log x)^{1-2\sqrt{\beta(1-\beta)}}\sqrt{\log_2 x}}
\end{split}
\end{equation*}
which is again negligible in comparison to the error term. (Here we have noted that $w = \sqrt{\beta(1-\beta)} \log_2 x \asymp \log_2 x$.) 

In the same range of $\beta$, we find that the interval $[w - w', w + w']$ gives the main contribution to the sum in \eqref{eq:Np1Expr1}. For $k$ in this range  we have $$k = w +O\left(\sqrt{w\log w}\right)=w\left(1+O\left(\sqrt{\frac{{\log_3 x}}{{\log_2 x}}}\right)\right),$$ and so throughout this range we find that %
$\xi = \frac k{\log u} \left(1-\frac \gamma{\log u}\right)^{-1} = \sqrt{\frac \beta{1-\beta}} + O\left(\sqrt{\frac{\log_3 x}{\log_2 x}}\right)$
and $\frac k{\log_2 p} = \sqrt{\frac{1-\beta}\beta} + O\left(\sqrt{\frac{\log_3 x}{\log_2 x}}\right)$.  From this a routine calculation shows that for all such $k$ we can write \[\eta\left(\frac k{\log_2 p}\right) \frac{k^{1/2} e^{-\gamma \xi}}{\Gamma(1+\xi)} = \frac{\eta\left(\sqrt{\frac{1-\beta}\beta}\right) \exp\left(-\gamma \sqrt{\frac \beta{1-\beta}}\right)\sqrt{w}}{ \Gamma\left(1+\sqrt{\frac \beta{1-\beta}}\right)}  \left(1+O\left(\sqrt{\frac{\log_3 x}{\log_2 x}}\right) \right).\] Consequently, for $\beta \in (1/5+\epsilon, 1-\epsilon)$, the contribution of $k \in [w-w', w+w']$ to $\Npox$ is   
\begin{align*}
\frac x{\sqrt{\pi} p \log x \log u} &\left(1 + O\left(\frac1{\sqrt{\log_2 x}}\right)\right) \sum_{w - w' \le k \le w + w'} \eta\left(\frac k{\log_2 p}\right) \frac{k^{1/2} e^{-\gamma \xi}}{\Gamma(1+\xi)} \frac{(2w)^{2k}}{(2k)!} \\
=&\ \frac{\eta\left(\sqrt{\frac{1-\beta}\beta}\right) e^{-\gamma \sqrt{\frac \beta{1-\beta}}}}{\sqrt{\pi}\  \Gamma\left(1+\sqrt{\frac \beta{1-\beta}}\right)} \frac{x\sqrt{w}}{p \log x \log u} \left(1+O\left(\sqrt{\frac{\log_3 x}{\log_2 x}}\right) \right) \sum_{w - w' \le k \le w + w'} \frac{(2w)^{2k}}{(2k)!}\nonumber \\
=&\ \frac{\overline{C_\beta} x}{p (\log x)^{1-2\sqrt{\beta(1-\beta)}} \sqrt{\log_2 x}} \left(1+O\left(\sqrt{\frac{\log_3 x}{\log_2 x}}\right) \right) \nonumber
\end{align*}
where we have used Lemma \ref{lem:NortonAppGen} to extend the sum over all values of $k$ and noted that $\sum\limits_{k=0}^\infty \frac{(2w)^{2k}}{(2k)!} = \frac12(e^{2w} + e^{-2w}) = \frac{e^{2w}}{2} + O(e^{-2w})$. %
We have thus established that
\begin{equation}\label{Np1Asymp}
N_{p, 1}(x) =
\begin{cases}
  \displaystyle{\frac{\overline C_\beta x}{p (\log x)^{1-2\sqrt{\beta(1-\beta)}} \sqrt{\log_2 x}} } \left(1+O\left(\sqrt{\frac{\log_3 x}{\log_2 x}}\right) \right) & \text{if  }\frac 15+\epsilon < \beta < 1-\epsilon,\\
\vspace{-2mm}\\
 O\left( \displaystyle{ \frac x{p (\log x)^{1/2-3\beta/2}(\log_2 x \log_3 x)^4}} \right)%
& \text{if  }\frac{3 \log_3 x}{\log_2 x} < \beta < \frac 15-\epsilon.
\end{cases}
\end{equation}
\medskip 
\noindent\textit{Estimation of $N_{p, 2}(x)$}: We recall that 
$$N_{p, 2}(x) = \frac x{p \log x} \left(1 + O\left(\frac1{\sqrt{\log_2 x}}\right)\right) \sum_{(2-\delta)\log_2 p < k \le K \log_2 x} \frac{e^{-\gamma \xi}}{\Gamma(1+\xi)} \frac{(\log u)^{k-1}}{(k-1)!} \sum_{\substack{P^+(A) \le p \\ \Omega(A) = k}} \frac 1A.$$
For $\beta \in (1/5+\epsilon, 1-\epsilon)$, we invoke Lemma \ref{lem:coarsebd} to obtain 
$$N_{p, 2}(x) \ll \frac x{p \log x} \sum_{(2-\delta)\log_2 p < k \le K \log_2 x} \frac{(\log u)^{k-1}}{(k-1)!} \frac k{2^k} \log^2 p \ll \frac{x \log_2 x \log^2 p}{p \log x} \sum_{k>(2-\delta)\log_2 p} \frac{v^k}{k!}$$
where $v\coloneqq\frac12 \log u \asymp \log_2 x$. Defining $\theta$ so that $v(1+\theta) = (2-\delta) \log_2 p$, we see that $\theta = \frac{2(2-\delta)}{1/\beta-1}-1 \asymp 1$. An application of \eqref{eq:Norton47} now yields
\begin{align}
N_{p, 2}(x) &\ll \frac{x \log_2 x \log^2 p}{p \log x} \cdot \frac{\exp\left(v(R(\theta)+1)\right)}{\sqrt v}\nonumber \\ &\ll \frac{x \sqrt{\log_2 x}(\log x)^{2\beta}\exp\left((\frac {1-\beta}2)(\frac{2(2-\delta)}{1/\beta-1})(1-\log\left(\frac{2(2-\delta)}{1/\beta-1}\right))\log_2 x\right)}{p \log x} \nonumber\\
&= \frac{x \sqrt{\log_2 x}}{p (\log x)^{1-2\sqrt{\beta(1-\beta)}+F(\beta, \delta)}}, \label{eq:Np2LargeBeta} 
\end{align}
where $F(\beta, \delta) \coloneqq 2\sqrt{\beta(1-\beta)} - (4-\delta) \beta + (2-\delta) \beta \log\left(\frac{2 \beta(2-\delta)}{1-\beta}\right)$. 

We claim that for any $\delta \ll_\epsilon 1$, $G(\delta) \coloneqq \inf_{\beta \in [1/5+\epsilon, 1-\epsilon]} F(\beta, \delta) > 0$. Indeed, since $F$ is continuous on $[1/5+\epsilon, 1-\epsilon] \times [0, 1]$, so is $G$ on $[0, 1]$; hence it suffices to show that $G(0) = \inf_{\beta \in [1/5+\epsilon, 1-\epsilon]} F(\beta, 0)>0$. But this in turn is an immediate consequence of the observation that $F\left(\frac15, 0\right)=0$ and that $$F(\beta,0) = 2\sqrt{\beta(1-\beta)}-4\beta+2\beta\log\left(\frac{4\beta}{1-\beta}\right)$$ is strictly increasing for $\beta>\frac 15$. This proves our claim. As such, for any fixed $\delta \ll_\epsilon 1$ and $c = c(\epsilon, \delta) \in (0, G(\delta)/2)$, we see that $F(\beta, \delta)>2 c$ for all $\beta \in (1/5+\epsilon, 1-\epsilon)$, %
leading to $$N_{p, 2}(x) \ll \frac{x}{p (\log x)^{1-2\sqrt{\beta(1-\beta)}+c}}$$ for all such $\beta$.

Now suppose $3\log_3 x/\log_2 x < \beta < 1/5-\epsilon$ and set $v'=2\sqrt{v\log v}$. We proceed as in \eqref{eq:NortonApp1} to bound the contributions of $k \in  ((2-\delta) \log_2 p, v - v'] \cup [v+v', K \log_2 x]$ in $N_{p,2}(x)$. Indeed, invoking \eqref{eq:coarsebd} %
to bound the sum of $1/A$, %
we see that this contribution is 
\begin{equation*}%
\begin{split}
&\ll \frac{x \log^2 p}{p \log x} \sum_{(2-\delta) \log_2 p \le k \le v - v'} \frac{(\log u)^{k-1}}{(k-1)!} \frac k{2^k} + \frac {x \log^2 p}{p \log x} \sum_{v + v' \le k \le K \log_2 x} \frac{(\log u)^{k-1}}{(k-1)!} \frac k{2^k}\\
&\ll \frac{x \log_2 x \log^2 p}{p \log x} \left(\sum_{k \le v - v'} \frac{v^k}{k!} + \sum_{k \ge v + v'} \frac{v^k}{k!}\right) \ll \frac1{\log_2 x\sqrt{\log_3 x}} \frac x{p(\log x)^{1/2-3\beta/2}}
\end{split}
\end{equation*}
which is negligible in comparison to the error term. Finally we use Theorem \ref{thm:ImprExtnLicht1} to estimate the contribution to $N_{p, 2}(x)$ from the range $(v- v', v + v')$.  To do so, we choose $\epsilon_0<\frac 12$ sufficiently small so that $(2+\sqrt{5\epsilon_0})\log_2 p < v-v'$. For sufficiently large $x$, the choice $\epsilon_0 = \epsilon^2$ suffices. Hence, the sought contribution is
\begin{equation*}
\begin{split}
\frac x{p \log x}& \left(1 + O\left(\frac1{\sqrt{\log_2 x}}\right)\right) \sum_{v-v' < k < v + v'} \frac{e^{-\gamma \xi}}{\Gamma(1+\xi)} \frac{(\log u)^{k-1}}{(k-1)!} \sum_{\substack{P^+(A) \le p \\ \Omega(A) = k}} \frac 1A\\
&= \frac{\eta_o(2) e^{-\frac{\gamma}2}}{2\Gamma(3/2)} \frac{x \log^2 p}{p \log x} \left(1+O\left(\sqrt{\frac{\log_3 x}{\log_2 x}} + \frac1{(\log p)^{\epsilon_0} \sqrt{\log_2 p}}\right) \right) \sum_{v- v' < k < v + v'} \frac{v^{k-1}}{(k{-}1)!} , 
\end{split} 
\end{equation*}
where we noted that for all $k$ in the above range, $\xi = 1/2+O(\sqrt{\log_3 x/\log_2 x})$, so that $e^{-\gamma \xi} = e^{-\gamma/2} (1 + O(\sqrt{\log_3 x/\log_2 x}))$ and $\Gamma(1+\xi) = \Gamma(3/2) (1 + O(\sqrt{\log_3 x/\log_2 x}))$. The constant in the last display above is exactly $\overline C$, while the sum on $v$ is $e^v(1 + O(1/v^2 \sqrt{\log v})) = (\log x)^{\frac12(1-\beta)}(1 + O(1/(\log_2 x)^2 \sqrt{\log_3 x}))$ by Lemma \ref{lem:NortonAppGen}. Collecting estimates, we have now shown that
\begin{equation}\label{Np2Asymp}
N_{p, 2}(x) =
\begin{cases}
\ \displaystyle{O\left(\frac x{p (\log x)^{1-2\sqrt{\beta(1-\beta)}+c}}\right)} & \text{if } \frac 15+\epsilon < \beta < 1-\epsilon,\\ 
 \displaystyle{\frac{\overline C x}{p(\log x)^{(1-3\beta)/2}} \left(1+O\left(\sqrt{\frac{\log_3 x}{\log_2 x}} + \frac{(\log_2 p)^{-1/2}}{(\log p)^{\epsilon_0}} \right)\right)} & \text{if } \frac{3 \log_3 x}{\log_2 x} < \beta < \frac 15-\epsilon.
\end{cases} 
\end{equation}
Since $N_p(x) = N_{p, 1}(x) + N_{p, 2}(x)$, \eqref{Np1Asymp} and \eqref{Np2Asymp} together complete the proof of Theorem \ref{thm:main} in the cases $\frac{3 \log_3 x}{\log_2 x} < \beta < 1/5-\epsilon$ and $1/5+\epsilon < \beta < 1-\epsilon$.

\medskip 

It remains to consider the case $\beta \le 3 \log_3 x/\log_2 x$, that is, $p \le \exp((\log_2 x)^3)$. In this case, an application of Theorem \ref{thm:Alladi2} to \eqref{NpFirstExpr} yields\footnote{Here it is important that our choice of $K$ was less than $2$.} 

\begin{equation}\label{eq:NpExpr1_SmallBeta}
N_p(x) = \frac x{p\log x} \left(1 + O\left(\frac{(\log_2 p)^2}{\log_2 x}\right)\right) \sum_{k \le K \log_2 x} \frac{g(p, \mu)}{\Gamma(1+\mu)} \frac{(\log_2 x)^{k-1}}{(k-1)!} \sum_{\substack{P^+(A) \le p \\ \Omega(A) = k}} \frac1A
\end{equation}
where we have noted, as before, that $\log(x/Ap) = \log x (1+O(\log_2 x \log p/\log x))$. Proceeding as in the case $3\log_3 x/\log_2 x < \beta < 1/5-\epsilon$, we see that the contribution of $k \le \frac12 \log_2 x - 3\sqrt{\log_2 x \log_3 x}$ and $k \ge \frac12 \log_2 x + 3\sqrt{\log_2 x \log_3 x}$ to \eqref{eq:NpExpr1_SmallBeta} is $\ll x/p(\log x)^{1/2-2\beta}(\log_2 x)^8$ which is absorbed in the error term since $(\log x)^{\beta/2} < (\log x)^{3\log_3 x/2\log_2 x} = (\log_2 x)^{3/2} \ll (\log_2 x)^4$. Finally, the contribution of the remaining $k$ is, by Theorem \ref{thm:ImprExtnLicht1}, equal to
\begin{equation*}
\begin{split}
&\begin{multlined}[t]
 \frac{\eta_o(2)}2 \frac{x\log^2 p}{p\log x} \left(1+O\left(\sqrt{\frac{\log_3 x}{\log_2 x}} + \frac1{(\log p)^{\epsilon_0} \sqrt{\log_2 p}}\right)\right)\\ \hspace{2cm} \times \sum_{\frac12 \log_2 x - 3\sqrt{\log_2 x \log_3 x} < k < \frac12 \log_2 x + 3\sqrt{\log_2 x \log_3 x}} \frac{g(p, \mu)}{\Gamma(1+\mu)} \frac{(\frac12 \log_2 x)^{k-1}}{(k-1)!}    
\end{multlined}\\
&=\begin{multlined}
 \frac{\eta_o(2) g(p, 1/2)}{2\Gamma(3/2)} \frac x{p(\log x)^{1/2-2\beta}} \left(1+O\left(\sqrt{\frac{\log_3 x}{\log_2 x}} + \frac1{(\log p)^{\epsilon_0} \sqrt{\log_2 p}}\right)\right) 
\end{multlined}
\end{split} 
\end{equation*}
by a final application of Lemma \ref{lem:NortonAppGen} and the observation that $\mu = 1/2 + O(\sqrt{\log_3 x/\log_2 x})$. By Mertens' Theorem, we have $g(p, 1/2) = e^{-\gamma/2}(\log p)^{-1/2} (1+O(1/\log p))$, completing the proof of Theorem \ref{thm:main}. 
\qedhere

\subsection{The middle prime factor when $\Omega(n)$ is even}

Recall that for those integers having an even number of prime factors, we define the middle prime factor to be the smaller of the two possible choices.  As such, %
in order to modify the proof of Theorem \ref{thm:main} to handle $\underline{M}_p(x)$, we note that if $\Omega(n)=2k$, and $P^{(\frac 12)}(n)=p$ then we may write $n=ApB$ where $p = p_k(n)$, $\Omega(A)=k-1$, $\Omega(B)=k$ and $P^+(A) \leq p \leq P^-(B)$. As such, the same arguments in Theorem \ref{thm:main} go through only by changing all the conditions $\Omega(A)=k$ to $\Omega(A)=k-1$. The only notable effect of this change in our arguments is that in the case $\beta>3\log_3 x/\log_2 x$, the expression for $N_{p, 1}(x)$ has $(\log_2 p)^{k-1}/(k-1)!$ in place of $(\log_2 p)^k/k!$, and so by a change of variable $k \mapsto k-1$, we obtain the following analogue of \eqref{eq:Np1Expr1}: 
\begin{multline*}
N_{p, 1}(x) = \frac x{\pi^{1/2} p \log x} \left(1 + O\left(\frac1{\sqrt{\log_2 x}} + \frac 1{\log_2 p}\right)\right)\\ \times \sum_{k \le (2-\delta) \log_2 p - 1} \eta\left(\frac k{\log_2 p}\right) \frac{k^{-1/2} e^{-\gamma \xi}}{\Gamma(1+\xi)} \frac{(2w)^{2k}}{(2k)!} \left(1+O\left(\frac1k\right)\right).
\end{multline*}
Here the contributions of $k \le \min\{w-w', (2-\delta)\log_2 p-1\}$ and of $k \in (w+w', (2-\delta)\log_2 p-1]$ can be bounded as before; starting with the trivial bound $k^{-1/2} \ll 1$, we see that this contribution is\footnote{In order to obtain the correct power saving in $\log_2 x$ for even $\Omega(n)$, it is important to truncate the sum on $k$ to $[w-a\sqrt{w\log w}, w+a\sqrt{w\log w}]$ for some fixed $a \ge 2$. Our choice $a\coloneqq3$ when defining $w'$ was convenient, but larger $a$ would work just as well.} 
\[\ll \frac1{(\log_2 x)^4 (\log_3 x)^5} \ \frac x{p(\log x)^{1-2\sqrt{\beta(1-\beta)}} \sqrt{\log_2 x}}.\] 
All other sums are handled exactly as in %
the proof of Theorem \ref{thm:main}, and the final result of having $\Omega(A)=k-1$ in place of $\Omega(A)=k$ is that our constants $\overline C_\beta$ and $\overline C$ (that had arisen for the exact middle prime factor) are multiplied by $\sqrt{(1-\beta)/\beta}$ and $2$ in the cases $\beta \in (1/5+\epsilon, 1-\epsilon)$ and $\beta \in (0, 1/5-\epsilon)$ respectively. 

\section{The $\alpha$-positioned prime factor} \label{sec:Alpha}

We now generalize Theorem \ref{thm:middle}, giving a proof of the more general Theorem \ref{thm:AlphaModified}.  
The proof closely follows the proof of Theorem \ref{thm:main}, and so in many places we only describe the extra ideas necessary in the more general case and elaborate only on the differences with the arguments of Theorem \ref{thm:main}. %
\begin{proof}[Proof of Theorem \ref{thm:AlphaModified}]
As before, we assume throughout that $\beta<1-\epsilon$. %
This time at the outset, we distinguish between the cases $\beta>3 \log_3 x/\log_2 x$ and $\beta \le 3 \log_3 x/\log_2 x$, considering first the former. With $K_0\coloneqq 2.04/\log 2$, Lemma \ref{lem:TEN} shows that the $n \le x$ with $p|n$ having $\Omega(n) \le k_\alpha \coloneqq \max\{2/\alpha, 2/(1-\alpha)\}$ or $\Omega(n) > K_0 \log_2 x$ give a negligible contribution to our count; hence it suffices to establish the claimed formulae for the number $\Npax$ of $n \le x$ with $\Omega(n) \in (k_\alpha, K_0 \log_2 x]$ and $P^{(\alpha)}(n)=p$. We factor each such $n$ with $\Omega(n)=k \in (k_\alpha, K_0 \log_2 x]$, uniquely as $n= ApB$ with $P^+(A) \le p \le P^-(B)$, $\Omega(A) = \ceilalphak-1$ and $\Omega(B) = k-\ceilalphak = \floorOmalphak$, and use Theorem \ref{thm:Alladi} to estimate the count of $B$'s given $A$. This yields 
\begin{equation}\label{eq:postAlladi1}
\Npax = \left(1+O\left(\frac1{\sqrt{\log_2 x}}\right)\right) \frac x{p\log x} \sum_{k_\alpha < k \le K_0 \log_2 x} \freGamxi \frac{(\log u)^{\floorOmalphak-1}}{(\floorOmalphak-1)!} \!\sum_{\substack{P^+(A) \le p\\\Omega(A) = \ceilalphak-1}} \! \frac1A   
\end{equation}
with $\xi = \floorOmalphak/(\log u-\gamma)$. As before, we write $\Npax = \Npaox + \Npatx$, 
where $\Npaox$ is defined by restricting the above sum to $k_\alpha < k \le \TDelAlp$ for some $\delta>0$ which will be fixed to be small enough in terms of $\epsilon$. 

\textit{Estimation of $\Npaox$}: By Theorem \ref{thm:lichtman}, 
\begin{multline}\label{eq:Np1Licht}
\Npaox = \left(1+O\left(\frac1{\sqrt{\log_2 x}} + \frac1{\log_2 p}\right)\right) \frac x{p\log x \log u \log_2 p}\\ \times \sum_{k_\alpha < k \le \TDelAlp} \eta\left(\frac{\ceilalphak-1}{\log_2 p}\right) \freGamxi \ceilalphak \floorOmalphak \binom k{\ceilalphak} \frac{(\log u)^{\floorOmalphak} (\log_2 p)^{\ceilalphak}}{k!}.
\end{multline}
Now since $k \ge k_\alpha$, we see that $0 \le \frac{\ceilalphak - \alpha k}{\alpha k} \le \frac 12$. Consequently, 
\begin{align*}
    \ceilalphak \log \ceilalphak &- \ceilalphak \log(\alpha k) \\&= \alpha k \left(1+\frac{\ceilalphak - \alpha k}{\alpha k}\right) \log \left(1+\frac{\ceilalphak - \alpha k}{\alpha k}\right) = \ceilalphak - \alpha k + O\left(\frac1k\right)
\end{align*}
and likewise \[\floorOmalphak \log \floorOmalphak - \floorOmalphak \log((1-\alpha) k) = \floorOmalphak - (1-\alpha)k + O\left(\frac 1k\right),\] so that, by Stirling's estimate, 
$$\binom k{\ceilalphak} = \frac{k^k \left(2 \pi \alpha (1-\alpha) k\right)^{-1/2}\OOk}{\ceilalphak^{\ceilalphak} \floorOmalphak^{\floorOmalphak}}  = \frac{\left(2 \pi \alpha (1-\alpha) k\right)^{-1/2}}{\alpha^{\ceilalphak} (1-\alpha)^{\floorOmalphak}} \OOk.$$
Consequently, \eqref{eq:Np1Licht} yields
\begin{multline}\label{eq:Np1PostBigStir}
\Npaox = \left(1+O\left(\frac1{\sqrt{\log_2 x}} + \frac1{\log_2 p}\right)\right) \frac1{\sqrt{2\pi \alpha (1-\alpha)}} \frac x{p\log x \log u \log_2 p}\\ \sum_{k_\alpha < k \le \TDelAlp} \eta\left(\frac{\ceilalphak-1}{\log_2 p}\right) \freGamxi \frac{\ceilalphak \floorOmalphak}{k^{1/2}} \frac{w^k}{k!} \chiexpfr \OOk    
\end{multline}
where we set $w \coloneqq Q_{\alpha,\beta} \log_2 x$ with $Q_{\alpha, \beta} \coloneqq \alphbetcom$. Note $w \gg \beta^\alpha \log_2 x \gg \log_2 p$.

Letting $E_w \coloneqq 6\sqrt{w \log w}$, we see that the total contribution of the terms $k \le \min\Big\{w-E_w,$ $\TDelAlp \Big\}$ to the expression in \eqref{eq:Np1PostBigStir} is, by Lemma \ref{lem:NortonAppGen},
\begin{equation}\label{eq:Np1Outsidemain}
\ll \frac{x\sqrt{\log_2 p}}{p\log x \log u} \sum_{k<w-E_w} \frac{w^k}{k!} \ll \frac{x\sqrt{\log_2 p}}{p(\log x)^{1-Q_{\alpha, \beta}} \log_2 x} \frac1{(\log_2 p)^{18}}   
\end{equation}
which is negligible in comparison to all the claimed error terms as $1-Q_{\alpha, \beta} \ge 1-2\beta-\nu^\alpha(1-\beta)$ for all $\alpha, \beta \in (0, 1)$.\footnote{This follows from the observation that for each value of $\alpha \in (0, 1)$, the function $H_{\alpha}(\beta) \coloneqq 2\beta+\nu^\alpha(1-\beta) - \alphbetcom \ge 0$ for all $\beta \in (0, 1)$. Indeed for all such $\beta$, $\frac{\partial^2 H_\alpha}{\partial \beta^2}>0$, so $H_{\alpha}(\beta)$ is convex on $(0, 1)$ and has a unique minimum in $(0, 1)$. Since $\frac{\partial H_\alpha}{\partial \beta} \big\vert_{\beta=\betalp} = 0 = H_{\alpha}(\betalp)$, it follows that $H_{\alpha}(\beta) \ge H_{\alpha}(\betalp) = 0$.} In particular, for $\beta < \betalp-\epsilon$, this shows that $\Npaox$ is absorbed in the error term. On the other hand, in the case $\beta \in (\betalpPl, 1-\epsilon)$, the contribution of $k \in (w+E_w, \TDelAlp]$ is bounded by the same expression as in \eqref{eq:Np1Outsidemain}, and thus is also negligible. Finally, in the same range of $\beta$ and for any $k \in [w-E_w, w+E_w]$, we have $k=w+O(E_w)$; hence, calculations analogous to those carried out for the exact middle prime factor reveal that the contribution of all such $k$ to the sum in \eqref{eq:Np1PostBigStir} is 
\begin{equation}\label{eq:Np1Main}
\left(1+O\left(\sqrt{\frac{\log_3 x}{\log_2 x}}\right)\right) \frac{C_{\beta, \alpha}}{\rho_{\chi, \alpha}} \frac x{p \log x \sqrt{\log_2 x}} \sum_{w-E_w \le k \le w+E_w} \frac{w^k}{k!} \chiexpfr.
\end{equation}
The corresponding assertions of Lemma \ref{lem:beattyModified} now show that in the case $\beta \in (\betalpPl, 1-\epsilon)$, $\Npaox$ satisfies the claimed asymptotic formulae for $\Npax$. The same estimate for $\Npax$ holds true uniformly for $\beta \in (\epsilon, 1-\epsilon)$ and $\alpha \in (\beta - \err, \beta+\err)$, in which case, for all sufficiently large $x$, we have $\alpha > \epsilon/2$ and $\beta>\beta_\alpha+\epsilon_1/2$ for some $\epsilon_1 >0$ depending only on $\epsilon$. Moreover in this case, we see that $C_{\beta, \alpha}/\rho_{\chi, \alpha} = C_{\beta, \beta} (1+O(\err))$, since $C_{\beta, \alpha} = \rho_{\chi, \alpha} H(\alpha, \beta)$ for some function $H(\alpha, \beta)$ differentiable on the compact set $[\epsilon/2, 1-\epsilon/2] \times [\epsilon, 1-\epsilon]$. %

\textit{Estimation of $\Npatx$}: To finish off the case $\beta>3\log_3 x/\log_2 x$, it remains to show that $\Npatx$ is negligible for $\beta > \betalpPl$ and that it satisfies the claimed asymptotic formulae for $\Mpax$ when $\beta \in (3\log_3 x/\log_2 x, \betalp-\epsilon)$. Indeed, for $\beta > \betalpPl$, \eqref{eq:coarsebd} shows that
\begin{equation*}
\begin{split}
\Npatx &\ll \frac {x \log^2 p}{p \log x} \sum_{\TDelAlp < k \le K_0 \log_2 x} \frac k{2^{\alpha k}} \cdot \frac{(\log u)^{\floorOmalphak-1}}{(\floorOmalphak-1)!}\\
&\ll \frac {x \log_2 x}{p (\log x)^{1-2\beta}} \sum_{\TDelAlp < k \le K_0 \log_2 x} \frac 1{2^{\alpha k}} \cdot \frac{(\log u)^{\floorOmalphak+1}}{(\floorOmalphak+1)!}\\
&\ll \frac {x \log_2 x}{p (\log x)^{1-2\beta}} \sum_{m>\TDelPar(1-\alpha)\log_2p} \frac{v^m}{m!}   
\end{split}
\end{equation*}
where we have defined $v\coloneqq \nu^\alpha \log u = \nu^\alpha (1-\beta) \log_2 x$ and set $m = \floorOmalphak + 1$ in the last equality above. (Here it is important that there are $\le 1/(1-\alpha) \ll 1$ many values of $k$ giving rise to a value of $m$.) Considering $\theta \in (0, 1)$ satisfying $(1+\theta)v \coloneqq \TDelPar(1-\alpha)\log_2p$, an application of \eqref{eq:Norton47} reveals (by a calculation analogous to \eqref{eq:Np2LargeBeta}) that $\Npatx \ll x\sqrt{\log_2 x} /p (\log x)^{1-Q_{\alpha, \beta}+F(\alpha, \beta, \delta)}$, where 
\begin{multline*}
F(\alpha, \beta, \delta) \coloneqq \alphbetcom - 2\beta \\- (2-\delta)\beta \left(\frac1\alpha-1\right) \left\{1-\log\left(\left(1-\frac\delta2\right)\frac{2^{1/(1-\alpha)}(1/\alpha-1)}{1/\beta-1}\right)\right\}.    
\end{multline*}
As such, it suffices to show that for any fixed $\epsilon_1>0$, there exists $\epsilon_2>0$ (depending at most on $\epsilon$ and $\epsilon_1$) such that $\displaystyle{G(\delta) \coloneqq \inf\limits_{(\alpha, \beta) \in [\epsilon_1, 1-\epsilon_1] \times [\betalpPl, 1-\epsilon]} F(\alpha, \beta, \delta)}$ $> 2\epsilon_2$ for all $\delta \ll_{\epsilon, \epsilon_1} 1$.\footnote{Notice that this includes both the cases of fixed and varying $\alpha$ considered in Theorem \ref{thm:AlphaModified}.} But since $F$ is continuous on $[\epsilon_1, 1-\epsilon_1] \times [\betalpPl, 1-\epsilon] \times [0, 1]$, so is $G$, and it suffices to show that $G(0)>0$ or (by compactness) that $F(\alpha, \beta, 0)>0$ for each $(\alpha, \beta) \in [\epsilon_1, 1-\epsilon_1] \times [\betalpPl, 1-\epsilon]$. This in turn follows from an analysis of the first two partial derivatives of $F$ with respect to $\beta$ on the interval $(\betalp, 1)$.\footnote{Indeed, it is easy to see that for each value of $\alpha \in (0, 1)$, we have $\frac{\partial^2 F}{\partial \beta^2}>0$ for all $\beta \in (\betalp, 1)$. Hence, $\frac{\partial F}{\partial \beta}$ an increasing function of $\beta$ on $(\betalp, 1)$. Since $\frac{\partial F}{\partial \beta}\big\vert_{\beta=\beta_\alpha} = 0$, it follows that $F(\alpha, \beta, 0)$ itself is an increasing function of $\beta$ on $(\betalp, 1)$. In particular, we have $F(\alpha, \beta, 0) > F(\alpha, \betalp, 0) = 0$, as desired.}  

Coming to the case $\beta \in (3\log_3 x/\log_2 x, \betalpM)$, a straightforward adaptation of the prior computations (invoking Lemmas \ref{lem:coarsebd} and \ref{lem:NortonAppGen}) shows, with $E_v \coloneqq 6\sqrt{v \log v}$, the contribution of $k \in \left(\TDelAlp,  \frac{v-E_v}{1-\alpha}\right) \cup \left(\frac{v+E_v}{1-\alpha}, K_0 \log_2 x\right]$ to $\Npatx$ is $\ll x/p (\log x)^{\SmallBetaExp} (\log_2 x)^{17}$, which is negligible in comparison to the claimed error (here we have noted that $v \asymp \log_2 x$). On the other hand, since $\beta<\betalpM$, we have $\frac{v-E_v}{1-\alpha} > \left(\frac{2+\epsilon_1}\alpha\right) \log_2 p$ for any %
$\epsilon_1 \ll_\epsilon 1$. Consequently, \eqref{eq:ImpExtLich1} shows that the contribution to $\Npatx$ from $k \in \left[ \frac{v-E_v}{1-\alpha}, \frac{v+E_v}{1-\alpha}\right]$ is
\begin{equation}\label{Np2Main}
\left(1+O\left(\sqrt{\frac{\log_3 x}{\log_2 x}} + \frac{(\log_2 p)^{-1/2}}{(\log p)^{\epsilon^2}}\right)\right) \frac{2\eta_o(2) \nu^\alpha \exp(-\gamma \nu^\alpha)}{\Gamma(1+\nu^\alpha)} \frac x{p(\log x)^{1-2\beta}} \cdot S(v), %
\end{equation}
where the sum $S(v)\coloneqq \sum_{\frac{v-E_v}{1-\alpha} \le k \le \frac{v+E_v}{1-\alpha}} \frac{v^{\floorOmalphak}}{\floorOmalphak!} \nuexpfr$ is estimated by the corresponding assertions of Lemma \ref{lem:beattyModified}. This completes the proof of the theorem for $\beta>3 \log_3 x/\log_2 x$.

It remains to consider the case $\beta \le 3\log_3 x/\log_2 x$. This time we start by fixing $\epsilon_1, \epsilon_2 \in (0, 1/2)$ which satisfy $(2-\epsilon_1)/(1-\alpha) > 2+\epsilon_2$,\footnote{It is clear that such $\epsilon_1, \epsilon_2>0$ can be fixed only in terms of $\alpha$ if $\alpha$ itself is fixed, or only in terms of $\epsilon$ if $\alpha \in (\epsilon/2, 1-\epsilon/2)$.} and removing the $n \le x$ divisible by $p$ which have $\Omega(n) > K_\alpha \log_2 x$ with $K_\alpha \coloneqq (2-\epsilon_1)/(1-\alpha)$; the number of such $n$ is $\ll x/p(\log x)^{K_\alpha \log 2 - 1}$, which is negligible in comparison to the error terms claimed for $\beta<\beta_\alpha - \epsilon$ (here, the constraint $\epsilon_1<1/2$ ensures that $K_\alpha \log 2 - 1 -(1-2\beta-2\nu(1-\beta)) > \epsilon'$ for some $\epsilon' \ll_\epsilon 1$). 
Hence, it suffices to show the claimed asymptotics for the number $\Npatilx$ of $n \le x$ having $\Omega(n) \le K_\alpha \log_2 x$ and $P^{(\alpha)}(n)=p$. Writing each such $n$ as $ApB$ exactly as before, Theorem \ref{thm:Alladi2} allows us to estimate the number of $B$ given $A$ (in order to apply the theorem, it is crucial that $(1-\alpha) K_\alpha$ is less than and bounded away from $2$). We deduce that
\begin{equation}\label{eq:Nptil_PostAlladi2}
\Npatilx = \left(1+O\left(\frac{(\log_2 p)^2}{\log_2 x}\right)\right) \frac x{p\log x} \sum_{k \le K_\alpha \log_2 x} \frac{g(p, \mu)}{\Gamma(1+\mu)} \frac{(\log_2 x)^{\floorOmalphak-1}}{(\floorOmalphak-1)!} \sum_{\substack{P^+(A)\le p\\ \Omega(A) = \ceilalphak-1}} \frac1A    
\end{equation}
with $\mu \coloneqq (\floorOmalphak-1)/\log_2 x$. Setting $v' \coloneqq \nu^\alpha \log_2 x \asymp \log_2 x$ and $E' \coloneqq 6\sqrt{v' \log v'}$, analogous calculations as before show that the $k < (v'-E')/(1-\alpha)$ and $k>(v'+E')/(1-\alpha)$ give a contribution $\ll x/p(\log x)^{1-2\beta-\nu^\alpha} (\log_2 x)^{17}$ to the sum in \eqref{eq:Nptil_PostAlladi2}, which is absorbed in the error terms since $(\log x)^{\nu^\alpha \beta} < (\log x)^{3 \log_3 x/\log_2 x} = (\log_2 x)^3 \ll (\log_2 x)^5$. Finally, since $\beta=o(1)$, we have $\alpha k > (2+\epsilon_1)\log_2 p$ for all $k \in \left[\frac{v'-E'}{1-\alpha}, \frac {v'+E'}{1-\alpha}\right]$. An application of Theorem \ref{thm:ImprExtnLicht1} reveals that the contribution of such $k$ to the sum \eqref{eq:Nptil_PostAlladi2} is
\begin{equation*}
= \left(1+O\left(\sqrt{\frac{\log_3 x}{\log_2 x}} + \frac{(\log_2 p)^{-1/2}}{(\log p)^{\epsilon^2}}\right)\right) \frac{2 \eta_o(2) \nu^\alpha g(p, \nu^\alpha)}{\Gamma(1+\nu^\alpha)} \frac x{p(\log x)^{1-2\beta}} \cdot S(v'),
\end{equation*}
where the sum $S(v')\coloneqq \sum\limits_{\frac{v'-E'}{1-\alpha} \le k \le \frac {v'+E'}{1-\alpha}} \frac{v'^{\floorOmalphak}}{\floorOmalphak !} \nuexpfr$ %
is estimated by Lemma \ref{lem:beattyModified}. This concludes the proof of Theorem \ref{thm:AlphaModified}, upon noting that $g(p, \nu^\alpha) = \frac{\exp(-\gamma \nu^\alpha)}{(\log p)^{\nu^\alpha}}$ $\left(1+O\left(\frac1{\log p}\right)\right)$.
\end{proof}

\section{Proof of Theorem \ref{thm:MpaBound}}

We claim the following bounds, which together imply the assertion of the theorem.
\begin{enumerate}
\item[(i)] Fix $\delta \in (0, \epsilon)$. %
We have uniformly for $\beta \in (\epsilon, 1-\epsilon)$ and $\alpha \in (0, \epsilon-\delta) \cup (1-\epsilon + \delta, 1)$,
\begin{equation}\label{eq:MpaSmallLargeAlpha}
\Mpax \ll \frac x{p(\log x)^c}    
\end{equation}
for some constant $c = c(\epsilon, \delta)>0$.\footnote{In the rest of this section, we shall write $C(\epsilon, \delta)$ to mean a constant $C$ depending on $\epsilon$ and $\delta$.}  

\item[(ii)] Fix $\delta, \delta_0 > 0$ satisfying%
$$0 < \delta \le \frac12\left(1-\betaepse-\frac1{2-2\nuepse}\right) < \epsilon \text{  and  }0< \delta_0 \le 1-2\nuepse - (\betaepse+\delta)(2-2\nuepse),$$
where $\nuepse \coloneqq 2^{-\frac1{1-\frac\epsilon8}} = 2^{-\frac8{8-\epsilon}}$. Then we have, uniformly for $\alpha \in \left(\epsff, 1-\epsff\right)$, 
\begin{align} \label{eq:MidAlp}
\Mpax \ll 
\begin{cases}
\displaystyle{\frac x{p(\log x)^{1-\alphbetcom} \sqrt{\log_2 x}}} & \text{ uniformly in } \beta \in \left(\betalp+\delta, 1-\delta\right),\\
\vspace{-2mm}\\
\displaystyle{\frac{x\log_2 x}{p (\log x)^{\delta_0}}} & \text{ uniformly in }\beta \in (\delta, \betalp+\delta].
\end{cases}
\end{align}
\end{enumerate}
Theorem \ref{thm:MpaBound} follows by invoking (i) for $\alpha \in (0, \frac {\epsilon}3) \cup (1-\frac {\epsilon}3, 1)$ along with (ii) for $\alpha \in (\frac {\epsilon}4, 1-\frac {\epsilon}4)$. Hence, it remains to prove (i) and (ii).

To show claim (i), we first remove all $n$ which either have $\Omega(n) \le (1-\delta/4)\log_2 x$ or $\Omega(n) \ge (1+\delta/4) \log_2 x$, noting that the number of such $n$ is $\ll x/p(\log x)^{c_1}$ for some constant $c_1 = c_1(\delta) > 0$. Indeed, with $E_0$ denoting the set of all primes, Mertens' second theorem shows that $E_0(x/p) = \sum_{\ell \le x/p} 1/\ell = \log_2 x + C_0 + o(1)$ for some absolute constant $C_0>0$. As such, any $n$ with $\Omega(n) \le (1-\delta/4)\log_2 x$ or $\Omega(n) \ge (1+\delta/4) \log_2 x$ can be written as $n= mp$ for some $m \le x/p$ either having $\Omega(m) \le (1-\delta/4)E_0(x/p)$ or having $\Omega(m) \ge (1+\delta/5)E_0(x/p)$, hence Lemma \ref{lem:HT} shows that 
$$\sum_{\substack{n \le x:p|n\\ \Omega(n) \le (1-\delta/4)\log_2 x}} 1 + \sum_{\substack{n \le x:p|n\\ \Omega(n) \ge (1+\delta/4)\log_2 x}} 1 \le \sum_{\substack{m \le x/p\\ \Omega(m) \le (1-\delta/4)E_0(x/p)}} 1 + \sum_{\substack{m \le x/p\\ \Omega(m) \ge (1+\delta/5)E_0(x/p)}} 1 \ll \frac x{p(\log x)^{c_1}}.$$
It thus remains to show that \eqref{eq:MpaSmallLargeAlpha} holds true for the count of $n \le x$ having $P^{(\alpha)}(n) = p$ and $(1-\delta/4)\log_2 x < \Omega(n) < (1+\delta/4)\log_2 x$; in the rest of the proof of claim (i), we only consider such $n$.

Suppose first that $\alpha \in (0, \epsilon-\delta)$. Then, since $p > \exp((\log x)^\epsilon)$, any such $n$ has at least $\lfloor (1-\alpha) \Omega(n) \rfloor + 1 > (1-\alpha) \Omega(n) > (1-\epsilon + \delta) (1-\delta/4)\log_2 x > (1-\epsilon + \delta) (1-\delta/2)\log_2 x + 1$ many prime divisors (counted with multiplicity) greater than $\explgxeps$. Hence, any such $n$ can be written as $n=mp$ for some $m \le x/p$ having $\Omega_E(m) > (1-\epsilon + \delta) (1-\delta/2)\log_2 x$, where $E$ denotes the set of primes exceeding $\explgxeps$. Since $E(x/p) = \sum_{\explgxeps < \ell \le x/p} 1/\ell = (1-\epsilon)\log_2 x + o(1) < (1-\epsilon + \delta/4) \log_2 x$, we obtain, by defining $\mu_\epsilon \coloneqq \frac{(1-\epsilon + \delta) (1-\delta/2)}{1-\epsilon + \delta/4} > 1$ and again applying Lemma \ref{lem:HT},
$$\sum_{\substack{n \le x: P^{(\alpha)}(n) = p\\\Omega(n) > (1-\delta/4)\log_2 x}} 1 \ \le \ \sum_{\substack{m \le x/p\\ \Omega_E(m) > \mu_\epsilon E(x/p)}} 1 \ \ll \  \frac x{p(\log x)^{c_1}}$$
for some constant $c_1 = c_1(\epsilon, \delta)>0$. This shows claim (i) for all $\alpha \in (0, \epsilon-\delta)$. 

Likewise for $\alpha \in (1-\epsilon + \delta, 1)$, since $p < \explgxoeps$, any $n$ with $\Omega(n) < (1+\delta/4) \log_2 x$ that is counted in $\Mpax$ has at most $\lfloor (1-\alpha) \Omega(n) \rfloor < (\epsilon-\delta) (1+\delta/4)\log_2 x$ many prime divisors (counting multiplicity) greater than $\explgxoeps$. Denoting by $E'$ the set of such primes, we see that $E'(x/p)>(\epsilon-\delta/4)\log_2 x$. Consequently, with $\nu_\epsilon \coloneqq \frac{(\epsilon-\delta)(1+\delta/4)}{\epsilon-\delta/4} \in (0, 1)$, Lemma \ref{lem:HT} yields
$$\sum_{\substack{n \le x: P^{(\alpha)}(n) = p\\\Omega(n) < (1+\delta/4)\log_2 x}} 1 \le \sum_{\substack{m \le x/p\\ \Omega_{E'}(m) < \nu_\epsilon E'(x/p)}} 1 \ \ll \  \frac x{p(\log x)^{c_2}}$$
for some constant $c_2 = c_2(\epsilon, \delta)>0$. This completes the proof of claim (i).

We now establish claim (ii) by closely following the proof of Theorem \ref{thm:AlphaModified}. To begin, defining $K_0\coloneqq 2.04/\log 2$ and $k_\epsilon \coloneqq 8/\epsilon = \max\{8/\epsilon, 8/(4-\epsilon)\}$, Lemma \ref{lem:TEN} again shows that the contribution of $n$ with $\Omega(n) \le k_\epsilon$ or $\Omega(n) > K_0\log_2 x$ are both negligible, making it sufficient to show the claim with $\Mpax$ replaced by the count $\Npax$ of $n \le x$ having $P^{(\alpha)}(n) = p$ and $\Omega(n) \in (k_\epsilon, K_0\log_2 x]$. Since $\alpha \in (\frac{\epsilon}4, 1-\frac{\epsilon}4)$, proceeding as in Theorem \ref{thm:AlphaModified}, we obtain 
\begin{equation*}%
\Npax = \left(1+O\left(\frac1{\sqrt{\log_2 x}}\right)\right) \frac x{p\log x} \sum_{k_\epsilon < k \le K_0 \log_2 x} \freGamxi \frac{(\log u)^{\floorOmalphak-1}}{(\floorOmalphak-1)!} \!\! \sum_{\substack{P^+(A) \le p\\\Omega(A) = \ceilalphak-1}} \! \frac1A.
\end{equation*}
Bounding the sum on $A$ by Lemma \ref{lem:coarsebd} and proceeding as before, we see that, uniformly for $\alpha \in (\frac {\epsilon}4, 1-\frac{\epsilon}4)$ and $\beta \in (\epsilon, 1-\epsilon)$, we have 
\begin{equation*}
\begin{split}
\Npax &\ll \frac x{p\log x} \sum_{k_\epsilon < k \le K_0 \log_2 x} \frac k{2^{\alpha k}} \log^2 p \cdot \frac{(\log u)^{\floorOmalphak-1}}{(\floorOmalphak-1)!}\\
&\ll \frac{x\log_2 x}{p(\log x)^{1-2\beta}} \sum_{k_\epsilon < k \le K_0 \log_2 x} \frac 1{2^{\alpha k}} \cdot \frac{(\log u)^{\floorOmalphak}}{\floorOmalphak!}\\
&\ll \frac{x\log_2 x}{p(\log x)^{1-2\beta}} \sum_{m \ge 1} \frac{(2\nu_\alpha \log u)^m}{m!} \ll \frac{x \log_2 x}{p(\log x)^{1-2\beta-2\nu_\alpha(1-\beta)}},
\end{split}
\end{equation*}
where $\nu_\alpha \coloneqq 2^{-1/(1-\alpha)}$ and we have set $m \coloneqq \floorOmalphak$, noting that there are $\le 1/(1-\alpha) \ll_\epsilon 1$ many possible values of $k$ corresponding to a given value of $m$. 
Now with $\delta$ chosen as in the statement of the claim, we see that $\delta<\frac{\epsilon}8 < \frac 1{16}$, so that the function $\alpha \mapsto 1-2\nu_\alpha - (\beta_\alpha+\delta)(2-2\nu_\alpha)$ is monotonically increasing on $(\frac{\epsilon}8, 1-\frac{\epsilon}8)$.\footnote{In fact, the function $\alpha \mapsto 1-2\nu_\alpha - (\beta_\alpha+\delta)(2-2\nu_\alpha)$ is increasing on $(0, 1)$, for each fixed $\delta \in (0, 1/16)$.} As such, for all $\beta \le \betalp + \delta$, the exponent of $\log x$ in the above display is $1-2\nu_\alpha - \beta(2-2\nu_\alpha) \ge 1-2\nu_\alpha - (\beta_\alpha+\delta)(2-2\nu_\alpha) > 1-2\nu_{\epse} - (\beta_{\epse}+\delta)(2-2\nu_{\epse}) \ge \delta_0$, showing the second assertion of claim (ii). 

Finally, in the case $\alpha \in (\frac{\epsilon}4, 1-\frac{\epsilon}4)$, $\beta \in (\betalp + \delta, 1-\delta)$, we can follow the proof of Theorem \ref{thm:AlphaModified} (in the case $\beta>\betalp + \epsilon$) more closely: writing $\Npax = \Npaox + \Npatx$ with the two summands defined analogously, we again see that $\Npatx$ is negligible in comparison to the error term, while the corresponding analogue of \eqref{eq:Np1PostBigStir} holds true for $\Npaox$. However at this point, invoking the trivial bound 
\begin{multline*}
\sum_{k_\alpha < k \le \TDelAlp} \eta\left(\frac{\ceilalphak-1}{\log_2 p}\right) \freGamxi \frac{\ceilalphak \floorOmalphak}{k^{1/2}} \frac{w^k}{k!} \chiexpfr\\ \ll (\log_2 x)^{3/2} \sum_{k \ge 1} \frac{w^k}{k!} \ll e^w (\log_2 x)^{3/2}
\end{multline*}
reveals that 
$$\Npaox \ll \frac x{p(\log x)^{1-\alphbetcom} \sqrt{\log_2 x}},$$
uniformly for $\alpha \in (\epsf, 1-\epsf)$ and $\beta \in (\betalp+\delta, 1-\delta)$. Hence the same bound holds for $\Npax$ as well. This establishes our claims, and completes the proof of the theorem.

\section{Proof of Theorem \ref{thm:RpNormal}}

We start by observing the following simple and useful bound on the tails of the Gaussian integral: for all $X>0$, we have
\begin{equation}\label{eq:GaussianTail}
\int_{-\infty}^{-X} e^{-u^2/2} \, \mathrm du = \int_X^\infty e^{-u^2/2} \, \mathrm du \le \frac1X \int_X^\infty u e^{-u^2/2} \, \mathrm du \ll \frac1Xe^{-X^2/2}.
\end{equation}
We first prove the theorem for $-\sqrt{\log_3 x} \leq t \leq \sqrt{\log_3 x}$. We claim that with 
\[\lambda \coloneqq \beta + \frac t{\sqrt{\log_2 x}},\] 
the left hand side of \eqref{eq:RpNormal} is\footnote{In fact, the arguments in the proof show that this identity holds true uniformly for all real $t$, but we shall not require this.}
\begin{equation}\label{eq:CountIntegral}
\frac px \sum_{\substack{n \le x: \, p|n\\R_p(n)<\lambda}} 1 = \frac{p \log_2 x}x \MpaxInt +   O\left(\frac1{(\log_2 x)^{1/3}}\right). 
\end{equation} 

To this end, we shall make frequent use of the following estimates
\begin{align}
\sum_{\substack{n \le x: \, p|n\\\Omega(n) \le \frac13\log_2 x}} 1 \ll \frac x{p(\log x)^{0.15}}, \label{eq:AtypicalOmega1}\\
\sum_{\substack{n \le x: \, p|n\\|\Omega(n) - \log_2 x| \ge (\log_2 x)^{2/3}}} 1 \ll \frac x{p(\log_2 x)^{1/3}}. %
\label{eq:AtypicalOmega2} 
\end{align}
The estimate \eqref{eq:AtypicalOmega1} is a direct consequence of Lemma \ref{lem:HT} since any $n \le x$ divisible by $p$ having $\Omega(n) \le \frac13\log_2 x$ is of the form $n= mp$ for $m \le x/p$ having $\Omega(m) \le \frac12 \log_2(x/p)$ (for all sufficiently large $x$). The estimates in \eqref{eq:AtypicalOmega2} follow from the Hardy-Ramanujan Theorem written in the form $\sum_{m \le x/p} \big(\Omega(m) - \log_2(x/p)\big)^2 \ll x \log_2 x/p$, since any $n$ counted in those two sums is of the form $mp$ for some $m \le x/p$ satisfying $|\Omega(m) - \log_2(x/p)| \gg (\log_2 x)^{2/3}$.

Turning now to the proof of the estimate \eqref{eq:CountIntegral}, we first note that since the number of $n \le x$ which are divisible by $p^2$ is $O(x/p^2)$, it suffices to show that \eqref{eq:CountIntegral} holds true with $\Mpax$ replaced by the count $\Mpasx$ of $n\le x$ exactly divisible by $p$ (meaning $p|n$ but $p^2  \nmid n$) for which $P^{(\alpha)}(n)=p$. Any such $n$ has $R_p(n) = \lceil \alpha \Omega(n) \rceil/\Omega(n)$, that is, $R_p(n)-1/\Omega(n) < \alpha \le R_p(n)$. Consequently, 
\begin{equation}\label{eq:MpasxIntExpr1}
\begin{split}
\int_0^\lambda \Mpasx \, \mathrm d\alpha &= \sum_{\substack{n \le x: \, p\parallel n\\R_p(n)<\lambda+1/\Omega(n)}} \int_{R_p(n)-1/\Omega(n)}^{\min\{\lambda, R_p(n)\}} \, \mathrm d\alpha\\ &= \sum_{\substack{n \le x: \, p\parallel n\\R_p(n)<\lambda}} \frac1{\Omega(n)} + \sum_{\substack{n \le x: \, p\parallel n\\\lambda \le R_p(n)<\lambda+1/\Omega(n)}} \left(\lambda - R_p(n) + \frac1{\Omega(n)} \right).
\end{split}
\end{equation}
For any $n$ counted in the second sum above, we have $|R_p(n)-\lambda|<1/\Omega(n)$, which shows that this sum is 
\begin{equation}\label{eq:LargeRp}
\begin{split}
\ll \sum_{\substack{n \le x: \, p\parallel n\\\lambda \le R_p(n)<\lambda+1/\Omega(n)}} \frac1{\Omega(n)} \ll \sum_{\substack{n \le x: \, p|n\\\Omega(n) \le \frac13\log_2 x}} 1 + \frac1{\log_2 x}\sum_{\substack{n \le x: \, p\parallel n\\\lambda \le R_p(n)<\lambda+1/\Omega(n)}} 1.
\end{split}
\end{equation}
By \eqref{eq:AtypicalOmega1}, the first of the two sums is $\ll x/p(\log x)^{0.15}$. Any $n$ counted in the second sum has $P^{(\lambda)}(n) = p$ and so, since $|\lambda - \beta| \le \sqrt{\frac{\log_3 x}{\log_2 x}}$, it follows from the final %
assertion of Theorem \ref{thm:AlphaModified} (for $\alpha \in (\beta - \err, \beta + \err)$), or from Theorem \ref{thm:MpaBound}, that $M^{\left(\lambda\right)}_{p}(x) \ll {x}/{p\sqrt{\log_2 x}}.$ %
Hence, the last expression in \eqref{eq:LargeRp} is $\ll x/p(\log_2 x)^{3/2}$, and \eqref{eq:MpasxIntExpr1} yields
\begin{equation}\label{eq:MpasxIntExpr2}
\int_0^\lambda \Mpasx \, \mathrm d\alpha = \sum_{\substack{n \le x: \, p\parallel n\\R_p(n)<\lambda}} \frac1{\Omega(n)} + O\left(\frac x{p(\log_2 x)^{3/2}}\right) = \sum_{\substack{n \le x: \, p\parallel n\\ \Omega(n) \in \mathcal W \\R_p(n)<\lambda}} \frac1{\Omega(n)} + O\left(\frac x{p(\log_2 x)^{4/3}}\right),
\end{equation}
where $\mathcal W \coloneqq (\log_2 x - (\log_2 x)^{2/3}, \log_2 x + (\log_2 x)^{2/3})$ and in the last equality above, we have invoked \eqref{eq:AtypicalOmega1} and \eqref{eq:AtypicalOmega2}. 

Finally, since any $n$ counted in the last sum in \eqref{eq:MpasxIntExpr2} has $\Omega(n) = \log_2 x(1+O((\log_ 2 x)^{-1/3}))$, it follows that the sum is 
\begin{equation*}
\begin{split}
&= \frac1{\log_2 x} \left(1+O\left(\frac1{(\log_ 2 x)^{1/3}}\right)\right) \sum_{\substack{n \le x: \, p\parallel n\\ \Omega(n) \in \mathcal W \\R_p(n)<\lambda}} 1 = \frac1{\log_2 x} %
\sum_{\substack{n \le x: \, p|n\\ R_p(n)<\lambda}} 1 + O\left(\frac x{p(\log_ 2 x)^{4/3}}\right)
\end{split}
\end{equation*}
by carrying out our earlier simplifications in reverse. Consequently, \eqref{eq:MpasxIntExpr2} yields 
$$\int_0^\lambda \Mpasx \, \mathrm d\alpha = \frac1{\log_2 x} \sum_{\substack{n \le x: \, p|n\\ R_p(n)<\lambda}} 1 + O\left(\frac x{p(\log_ 2 x)^{4/3}}\right),$$
establishing our claim \eqref{eq:CountIntegral} uniformly for all $t \in [-\sqrt{\log_3 x}, \sqrt{\log_3 x}]$. Hence in order to complete the proof of the theorem for all $t$ in this range, it suffices to show that 
\begin{equation}\label{eq:RpnIntAssertion}
\int_0^\lambda \Mpax \dalp = \frac x{p \log_2 x}  \left\{\Phi\left(\frac t{\sqrt{\beta(1-\beta)}}\right)+O\left(\frac{(\log_3 x)^{3/2}}{(\log_2 x)^{1/2}}\right)\right\}
\end{equation}
uniformly for all such $t$. 

Now for $t \in [-\sqrt{\log_3 x}, \sqrt{\log_3 x}]$, we have $\beta-\err \leq \lambda \leq \beta+\err$ where $\err \coloneqq \sqrt{\log_3 x/\log_2 x}$. Furthermore, for all $\alpha \in (\beta -\err,\beta+\err)$ we again find ourselves in the final case of Theorem \ref{thm:AlphaModified}, an application of which yields, 
\begin{equation*}
\int_{\beta-\err}^{\lambda} \Mpax \dalp = \left(1 + O\left(\sqrt{\frac{\log_3 x}{\log_2 x}}\right)\right) \frac{C_{\beta, \beta} x}{p \log x \sqrt{\log_2 x}} \int_{\beta-\err}^{\lambda} (\log x)^{\alphbetcom} \dalp .
\end{equation*}
To analyze the integral above, we note that for all $\alpha \eqqcolon \beta+\eta \in [\beta-\err, \beta+\err]$,
\begin{equation*}
\begin{split}
\alpha\log\left(\frac\beta\alpha\right)+(1-\alpha)\log\left(\frac{1-\beta}{1-\alpha}\right) &= -(\beta+\eta)\log\left(1+\frac\eta\beta\right)-(1-\beta-\eta)\log\left(1-\frac{\eta}{1-\beta}\right)\\
&= -\frac{\eta^2}{2\beta(1-\beta)} + O(\eta^3),
\end{split}
\end{equation*}
so that $\alphbetcom = 1-\frac{\eta^2}{2\beta(1-\beta)} + O(\eta^3)$, leading to
\begin{equation}\label{eq:logpowerEstim}
(\log x)^{\alphbetcom} = (\log x) \exp\left(-\frac{\eta^2 \log_2 x}{2\beta(1-\beta)}\right) \left(1+O\left(\frac{(\log_3 x)^{3/2}}{(\log_2 x)^{1/2}}\right)\right).
\end{equation}
As such, 
\begin{align*}
\int_{\beta-\err}^{\lambda} \Mpax \dalp &= \left(1 + O\left(\frac{(\log_3 x)^{3/2}}{(\log_2 x)^{1/2}}\right)\right) \frac{C_{\beta, \beta} x}{p \sqrt{\log_2 x}} \int_{-\err}^{\lambda-\beta} \exp\left(-\frac{\eta^2 \log_2 x}{2\beta(1-\beta)}\right) \deta\\
&= \left(1+O\left(\frac{(\log_3 x)^{3/2}}{(\log_2 x)^{1/2}}\right)\right) \frac x{p \log_2 x} \cdot \frac1{\sqrt{2\pi}} \int_{-\sqrt{\frac{\log_3x}{\beta(1-\beta)}}}^{ \frac t{\sqrt{\beta(1-\beta)}}} \exp\left(-\frac{\tau^2}2\right) \, \mathrm d\tau .
\end{align*}
Invoking the bound \eqref{eq:GaussianTail}, we obtain, uniformly for $t \in [-\sqrt{\log_3 x}, \sqrt{\log_3 x}]$,
\begin{equation}\label{eq:MainTermt>=0}
\int_{\beta-\err}^{ \beta+t/\sqrt{\log_2 x}} \Mpax \dalp = \frac x{p \log_2 x}  \left\{\Phi\left(\frac t{\sqrt{\beta(1-\beta)}}\right)+O\left(\frac{(\log_3 x)^{3/2}}{(\log_2 x)^{1/2}}\right)\right\},
\end{equation}
where we have noted that $\exp\left(-\frac{\log_3 x}{2\beta(1-\beta)}\right) \ll \frac1{(\log_2 x)^2}$, since $\beta(1-\beta) \le 1/4$.

We now show that the contribution to the integral in \eqref{eq:RpnIntAssertion} from $\alpha$ outside $[\beta-\err, \beta+t/\sqrt{\log_2 x}]$ is negligible.  To that end, we start by noting that by an argument analogous to the above, we have 
\begin{align}\label{eq:Int01Mpax}
\int_0^1 \Mpax \dalp = \frac x{p\log_2 x} + O\left(\frac x{p(\log_2 x)^2}\right)
\end{align}
uniformly for $\beta \in (\epsilon, 1-\epsilon)$. Indeed, it again suffices to show the above to be true with $\Mpax$ replaced by $\Mpasx$, and by a computation similar to \eqref{eq:MpasxIntExpr1}, we find that
\begin{align*}
\int_0^1 \Mpasx \dalp &= \sum_{\substack{n \le x\\p \parallel n}} \frac1{\Omega(n)} = \sum_{\substack{n \le x\\p | n}} \frac1{\Omega(n)} + O\left(\frac x{p^2}\right) = \sum_{\substack{m \le x/p}} \frac1{\Omega(m)+1} + O\left(\frac x{p^2}\right) \nonumber\\
&= \sum_{\substack{1<m \le x/p}} \frac1{\Omega(m)} + O\left(\sum_{1<\substack{m \le x/p}} \frac1{\Omega(m)^2} + \frac x{p^2}\right) = \frac x{p\log_2 x} + O\left(\frac x{p(\log_2 x)^2}\right)
\end{align*}
where in the last equality, we have invoked Theorems 5 and 14 in \cite{dekoninck72}, in the weak forms $\sum_{1<n \le x} 1/\Omega(n) = x/\log_2 x + O(x/(\log_2 x)^2)$ and $\sum_{1<n \le x} 1/\Omega(n)^2 \ll x/(\log_2 x)^2$ respectively.\footnote{Alternatively, we may replicate the arguments used to handle the sum $\sum_{\substack{n \le x: \, p\parallel n\\R_p(n)<\lambda}} \frac1{\Omega(n)}$ in \eqref{eq:MpasxIntExpr1}; this gives an error term of $O(x/p(\log_2 x)^{4/3})$, which is sufficient for the theorem.}  

Comparing \eqref{eq:MainTermt>=0}  with \eqref{eq:Int01Mpax} and taking $t = \sqrt{\log_3 x}$ we obtain 
\begin{align*}%
\int_0^{\beta-\err} &\Mpax \dalp + \int_{\beta+\err}^1 \Mpax \dalp \nonumber\\
&= \frac {x}{p\log_2 x} \left\{1-\Phi\left(\sqrt{\frac {\log_3 x}{\beta(1-\beta)}}\right)\right\} + O\left(\frac{x(\log_3 x)^{3/2}}{p(\log_2 x)^{3/2}} \right) \ll \frac{x(\log_3 x)^{3/2}}{p(\log_2 x)^{3/2}}, 
\end{align*}
where in the last equality, we have again applied \eqref{eq:GaussianTail}. Combining \eqref{eq:MainTermt>=0} with the bound above for $\int_0^{\beta-\err} \Mpax \dalp $ %
yields the desired relation \eqref{eq:RpnIntAssertion} for all $t \in [-\sqrt{\log_3 x}, \sqrt{\log_3 x}]$, proving the theorem for $t$ in this range.

Now suppose $t\leq -\sqrt{\log_3 x}$. We claim that both sides of \eqref{eq:RpNormal} are absorbed in the error term $O(1/(\log_2 x)^{1/3})$. This is immediate for the right hand side since by \eqref{eq:GaussianTail}, we have
$$\Phi\left(\frac t{\sqrt{\beta(1-\beta)}}\right) \le \Phi\left(-\sqrt{\frac{\log_3 x}{\beta(1-\beta)}}\right) \ll \frac1{(\log_2 x)^2 \sqrt{\log_3 x}}.$$
On the other hand, the left hand side of \eqref{eq:RpNormal} is
\[\frac px \!\! \sum_{\substack{n \le x: \, p|n\\R_p(n)<\beta+\frac{t}{\sqrt{\log_2 x}}}} \!\! 1 \ \le \ \frac px \!\! \sum_{\substack{n \le x: \, p|n\\R_p(n)<\beta - \frac{\sqrt{\log_3 x}}{\sqrt{\log_2 x}}}} \!\! 1 \ = \  \Phi\left(-\sqrt{\frac{\log_3 x}{\beta(1-\beta)}}\right) + O\left(\frac1{(\log_2 x)^{1/3}}\right) \ll \frac1{(\log_2 x)^{1/3}}, \] %
where we have used the assertion of the theorem for $t=-\sqrt{\log_3 x}$, that we established before. %

Finally, for $t \ge \sqrt{\log_3 x}$, the reasoning is analogous: by \eqref{eq:GaussianTail}, the right hand side of \eqref{eq:RpNormal} is %
\begin{multline*}
1+O\left(\frac1t \exp\left(-\frac{t^2}{2\beta(1-\beta)}\right) + \frac1{(\log_2 x)^{1/3}}\right)\\
= 1+O\left(\frac1{\sqrt{\log_3 x}} \exp\left(-\frac{\log_3 x}{2\beta(1-\beta)}\right) + \frac1{(\log_2 x)^{1/3}}\right) = 1 + O\left(\frac1{(\log_2 x)^{1/3}}\right),
\end{multline*}
whereas by the $t=\sqrt{\log_3 x}$ case of the theorem (that we already proved), 
\begin{align*}
1 \ge \frac px \sum_{\substack{n \le x: \, p|n\\R_p(n)<\beta+\frac{t}{\sqrt{\log_2 x}}}} 1 &\ge \frac px \sum_{\substack{n \le x: \, p|n\\R_p(n)<\beta  +\frac{\sqrt{\log_3 x}}{\sqrt{\log_2 x}}}} 1\\ &= \Phi\left(\sqrt{\frac{\log_3 x}{\beta(1-\beta)}}\right) + O\left(\frac1{(\log_2 x)^{1/3}}\right) = 1 + O\left(\frac1{(\log_2 x)^{1/3}}\right) 
\end{align*}
showing that the left hand side is also $1+ O\left(\frac1{(\log_2 x)^{1/3}}\right)$.
This completes the proof. \qedhere

\section{Proof of Theorem \ref{thm:avglogmid}}
We start by writing \begin{align}
\frac 1x \sum_{n \le x} \log P^{\left(\frac 12 \right)}(n) = \frac 1x\sum_{p \le x}  M^{\left( \frac 12 \right)}_p(x)\log p. \label{eq:logmidprime}
\end{align}
The trivial bound $M^{\left( \frac 12 \right)}_p(x)\leq \frac{x}{p}$ shows that the contribution to the above sum from $p \le \exp\left(\sqrt{\log x}\right)$ is
\[\frac 1x\sum_{p \le \exp(\sqrt{\log x})}  M^{\left( \frac 12 \right)}_p(x) \log p < \sum_{p \le \exp(\sqrt{\log x})} \frac{\log p}p \ll \sqrt{\log x},\]
and for $\sqrt{x}<p\leq x$ we have $M^{\left( \frac 12 \right)}_p(x)=1$ showing that the contribution from such $p$ is $O(1)$.  Next we bound the contribution from those $p$ with $\exp\left( (\log x)^{0.999}\right)<p\leq \sqrt{x}$.  

The same arguments as in the proof of Theorem \ref{thm:MpaBound} (claim (i)) show that $M^{\left( \frac 12 \right)}_p(x)\ll x/p(\log x)^{0.42}$ uniformly for $p \in (\exp((\log x)^{0.999}), \sqrt x]$: indeed once again, any $n$ counted in $M^{\left( \frac 12 \right)}_p(x)$ either has $\Omega(n) \le 0.229 \log_2 x$ or has more than $\frac{0.229}2 \log_2 x$ many prime factors (counted with multiplicity) exceeding $\exp((\log x)^{0.999})$. As such, two applications of Lemma \ref{lem:HT} with the set $E$ being the full set of primes or the set of primes exceeding $\exp((\log x)^{0.999})$ respectively, show that the contribution of both of these possibilities is $\ll x/p(\log x)^{0.42}$, as desired. Consequently, the total contribution from the primes $p \in (\exp((\log x)^{0.999}), \sqrt x]$ is
$$\frac 1x\sum_{\exp((\log x)^{0.999})<p\leq \sqrt{x}}  M^{\left( \frac 12 \right)}_p(x)\log p \ll \frac 1{(\log x)^{0.42}}\sum_{p\le x} \frac{\log p}p \ll (\log x)^{0.58},$$
which is also negligible for our purposes. 

It thus remains to consider primes $\exp\left(\sqrt{\log x}\right) < p \leq \exp\left( (\log x)^{0.999}\right)$, for which we can use Theorem \ref{thm:middle} to estimate $M^{\left( \frac 12 \right)}_p(x)$.  Thus the sum \eqref{eq:logmidprime} for $p$ in this range is
\begin{align}
\left(1+O\left(\sqrt{\frac{\log_3 x}{\log_2 x}}\right)\right)\sum_{\exp\left(\sqrt{\log x}\right) < p \leq \exp\left( (\log x)^{0.999}\right)}  C_\beta \displaystyle{\frac{(\log x)^{\beta + 2\sqrt{\beta(1-\beta)}-1}}{p\sqrt{\log_2 x}}} \label{eq:rewrittenlogsum}
\end{align}
where $\beta=\frac{\log_2 p}{\log_2 x}$ is a function of $p$.  Rewriting the sum above as an integral and using the prime number theorem we have 
\begin{align}
 \sum_{\exp\left(\sqrt{\log x}\right) < p \le \exp\left( (\log x)^{0.999}\right)} \!\!\!\!\! &C_\beta \displaystyle{\frac{(\log x)^{\beta + 2\sqrt{\beta(1-\beta)}-1}}{p\sqrt{\log_2 x}}} = \int_{\exp\left(\sqrt{\log x}\right)}^{\exp\left( (\log x)^{0.999}\right)} \!\! C_\beta \displaystyle{\frac{(\log x)^{\beta + 2\sqrt{\beta(1-\beta)}-1}}{t\sqrt{\log_2 x}}} \mathrm d \pi(t) \nonumber
    \\
    = \int_{\exp\left(\sqrt{\log x}\right)}^{\exp\left( (\log x)^{0.999}\right)} &C_\beta \displaystyle{\frac{(\log x)^{\beta + 2\sqrt{\beta(1-\beta)}-1}}{t \log t \sqrt{\log_2 x}}}\mathrm d t \ + \ O\!\left(\exp(-K_0(\log x)^{1/4})\right)\label{eq:SumIntRel1}
\end{align}
for some absolute constant $K_0>0$, where in the integrals we have $\beta \coloneqq \log_2 t/\log_2 x$. Here to pass to the second line above, we have noted that writing $f(t) \coloneqq C_\beta \displaystyle{\frac{(\log x)^{\beta + 2\sqrt{\beta(1-\beta)}-1}}{t\sqrt{\log_2 x}}}$ and $\err(t) \coloneqq \pi(t) - \text{li}(t) \ll t\exp(-K\sqrt{\log t})$, the function $f(t)$ is monotonically decreasing for all sufficiently large $x$ (and for $t \ge \exp(\sqrt{\log x})$),\footnote{To see this, we write $\log f(t) \eqqcolon F_1(\beta) + F_2(\beta) \log_2 x - \log t - \frac12\log_3 x$ for certain differentiable functions $F_1, F_2$, and differentiate both sides of this identity with respect to $t$ to obtain $f'(t)/f(t) < -1/t + O(1/t\log t) < 0$ for all sufficiently large $x$, uniformly in $\beta = \log_2 t/\log_2 x \in [0.5, 0.999]$.} so that two applications of the Riemann-Stieltjes integration by parts yield (with $K_0 \coloneqq K/2$),
\begin{align*}
\int_{\exp\left(\sqrt{\log x}\right)}^{\exp\left( (\log x)^{0.999}\right)} f(t) \, \mathrm d\err(t) &= -\int_{\exp\left(\sqrt{\log x}\right)}^{\exp\left( (\log x)^{0.999}\right)} \err(t) f'(t) \, \mathrm dt + O\!\left(\exp(-K_0(\log x)^{1/4})\right)\\
&\ll \int_{\exp\left(\sqrt{\log x}\right)}^{\exp\left( (\log x)^{0.999}\right)} t\exp(-K \sqrt{\log t}) f'(t) \, \mathrm dt + \exp(-K_0 (\log x)^{1/4})\\
& \ll \int_{\exp\left(\sqrt{\log x}\right)}^{\exp\left( (\log x)^{0.999}\right)}  f(t) \exp(-K \sqrt{\log t}) \mathrm dt + \exp(-K_0(\log x)^{1/4})\\
&\ll \exp(-K_0 (\log x)^{1/4}),
\end{align*}
establishing \eqref{eq:SumIntRel1}. Continuing from there, we find that the expression in \eqref{eq:rewrittenlogsum} is
\[\left(1+O\left(\sqrt{\frac{\log_3 x}{\log_2 x}}\right)\right) \int_{\exp\left(\sqrt{\log x}\right)}^{\exp\left( (\log x)^{0.999}\right)} \!\! C_\beta \displaystyle{\frac{(\log x)^{\beta + 2\sqrt{\beta(1-\beta)}-1}}{t \log t \sqrt{\log_2 x}}}\mathrm d t + O\!\left(\exp(-K_0 (\log x)^{1/4})\right).\]
Changing the variable of integration to $\beta$, using $t=\exp\left( (\log x)^\beta\right)$ we find that $\mathrm dt = t \log t \log_2 x \ \mathrm d\beta$ and so the main term above becomes
\begin{align} \label{eq:betavarchange} 
    \left(\sqrt{\log_2 x} +O\left(\sqrt{\log_3 x}\right)\right) \int_{1/2}^{0.999} &  C_\beta (\log x)^{\beta + 2\sqrt{\beta(1-\beta)}-1} \mathrm d\beta.
\end{align}
We find that the exponent of $\log x$ in the integrand above achieves its maximum at $B_0 \coloneqq \frac{1}{2} + \frac{1}{2\sqrt{5}}$, and its value at that point is $B_0+2\sqrt{B_0(1-B_0)}-1 = \frac{\sqrt{5}-1}{2} = \varphi'$. 

Furthermore, defining $\beta' \coloneqq \beta-B_0$, and expanding as a Taylor series around $B_0$ gives \begin{equation}
    \beta + 2\sqrt{\beta(1-\beta)}-1 = \varphi' -\frac {5\sqrt{5}}4 \beta'^2 +O\left(\beta'^3\right). \label{eq:betataylor}
\end{equation}

Inserting this into \eqref{eq:betavarchange} and treating first only the range $\beta \in \left[B_0 -\sqrt{\frac{\log_3 x}{\log_2 x}}, B_0 +\sqrt{\frac{\log_3 x}{\log_2 x}}\right]$, we find that
\begin{align*}
 \int_{B_0 -\sqrt{\frac{\log_3 x}{\log_2 x}}}^{B_0 +\sqrt{\frac{\log_3 x}{\log_2 x}}}&   C_\beta (\log x)^{\beta + 2\sqrt{\beta(1-\beta)}-1} \, \mathrm d\beta \\
 &=  \left(C_{B_0}+O\left(\sqrt{\frac{\log_3 x}{\log_2 x}}\right)\right) \int_{-\sqrt{\frac{\log_3 x}{\log_2 x}}}^{\sqrt{\frac{\log_3 x}{\log_2 x}}}  \exp\left(\left(\varphi'-\tfrac {5\sqrt{5}}4 \beta'^2 +O\left(\beta'^3\right)\right)\log_2 x\right) \mathrm d\beta' \\
 &=  \frac{C_{B_0}(\log x)^{\varphi'}}{\sqrt{\tfrac {5\sqrt{5}}4 \log_2 x}} \left(1+O\left(\frac{(\log_3 x)^{3/2}}{\sqrt{\log_2 x}}\right)\right) \int_{-\sqrt{\tfrac {5\sqrt{5}}4 \log_3 x}}^{\sqrt{\tfrac {5\sqrt{5}}4 \log_3 x}} \,\, \exp\left(-\tau^2 \right) \mathrm d\tau\\ 
 &=  C_{B_0}(\log x)^{\varphi'} \sqrt{\frac{\pi}{\frac{5\sqrt{5}}{4}\log_2 x}}\left(1+O\left(\frac{(\log_3 x)^{3/2}}{\sqrt{\log_2 x}} + \frac{{\exp\left(-\frac{5\sqrt 5}{4}\log_3 x\right)}}{\sqrt{\log_3 x}}\right)\right) \\ 
 &=  \left(C_{B_0}+O\left(\frac{(\log_3 x)^{3/2}}{\sqrt{\log_2 x}}\right)\right) (\log x)^{\varphi'} \sqrt{\frac{4\pi}{5\sqrt{5}\log_2 x}}.
 \end{align*}
Substituting the value of $C_{B_0}$ from \eqref{eq:Cbetaval}, we deduce that the contribution of $\beta \in \Big[B_0 -\sqrt{\frac{\log_3 x}{\log_2 x}},$ $B_0 +\sqrt{\frac{\log_3 x}{\log_2 x}}\Big]$ to the expression \eqref{eq:betavarchange} is
 \begin{align*}
    \left(\frac{e^{-\gamma}}{\Gamma(\varphi+1)}\frac{\varphi + 1}{\sqrt{5}}\prod_p\left(1-\frac{1}{p}\right)^{\varphi'}\left(1-\frac{\varphi'}{ p }\right)^{-1} +O\left(\frac{(\log_3 x)^{3/2}}{\sqrt{\log_2 x}}\right)\right) (\log x)^{\varphi'}. 
  \end{align*} 
We conclude by bounding the contribution from $\beta \in \left[\frac12, B_0-\sqrt{\tfrac{\log_3 x}{\log_2 x}}\right)$ and from $\beta \in \left(B_0+\sqrt{\tfrac{\log_3 x}{\log_2 x}}, 0.999\right]$ to the expression \eqref{eq:betavarchange}. Noting that the function $\beta \mapsto \beta + 2\sqrt{\beta(1-\beta)}-1$ is increasing on $\left[\frac12, B_0-\sqrt{\tfrac{\log_3 x}{\log_2 x}}\right)$ and then using the expansion \eqref{eq:betataylor}, we deduce that the contribution from this interval is 
\begin{align*}
\ll \sqrt{\log_2 x}\int_{1/2}^{B_0 -\sqrt{\frac{\log_3 x}{\log_2 x}}}  &C_\beta (\log x)^{\beta + 2\sqrt{\beta(1-\beta)}-1} \, \mathrm d\beta \\
&\ll (\log x)^{\varphi'} \exp\left(-\tfrac {5\sqrt{5}}4 \log_3 x\right)\sqrt{\log_2 x} \ \leq \ \frac{(\log x)^{\varphi'}}{(\log_2 x)^2}
\end{align*}
Finally, the contribution of the interval $\left(B_0+\sqrt{\tfrac{\log_3 x}{\log_2 x}}, 0.999\right]$ can be bounded analogously, by noting that the function $\beta \mapsto \beta + 2\sqrt{\beta(1-\beta)}-1$ is decreasing on this interval. This completes the proof of the theorem. 

\section{Concluding Remarks}

While we were able to obtain asymptotic expressions for the frequency with which $p$ is the middle or $\alpha$-positioned prime factor of an integer up to $x$ for a very wide range of values of $p$ (depending on $x$), our theorems don't quite encompass the full range of $p$.  

In particular, we are unable to treat those primes $p$ for which $\beta$ is too close to either $\beta_\alpha$ or to 1.  (We also don't consider primes $p$ fixed as $x\to \infty$, however this is treated, for the middle prime factor, in \cite{DoyOue18}.)  

The obstacles to understanding the behavior in these two remaining ranges are somewhat different.  The complication when $\beta \approx \beta_\alpha$ comes from the phase transition that occurs when $k \approx 2 \log_2 y$ in the asymptotics of the sums $\sum_{\substack{P^+(n)\le y\\\Omega(n)=k}} \frac{1}{n}$. (See Theorems \ref{thm:lichtman} and \ref{thm:ImprExtnLicht1}.)  Extending our results to $\beta$ in this range would require obtaining asymptotic expressions for these sums near this phase transition.  Balazard, Delange and Nicholas \cite{BDN} have investigated the closely related phase transition that happens in the counting problem of $N(x,k)=\#\{n\leq x \mid \Omega(n)=k\}$ for $k$ near $2\log_2 x$, and found that for such $k$ the correct asymptotic expressions for $k$ near the phase transition is a ``Gaussian transition'' between the asymptotic expressions valid for $k$ bounded away on either side of this transition value.  It seems likely that a similar transition happens in this case, however we don't investigate this any further here. 
Investigating large values of $p$ when $\beta$ approaches 1 appears to be more difficult.  Here the obstacle is the range of validity of Alladi's result (Theorem \ref{thm:Alladi}).  Extending our results would require obtaining an asymptotic expression for $\Phi_k(x,y)$ that holds for values of $k$ that are large relative to $\log u$, where $u =\frac{\log x}{\log y}$.  %
 
\bibliographystyle{amsplain}
\bibliography{refs}

\begin{bibdiv}
\begin{biblist}

\bib{Alladi}{article}{
      author={Alladi, Krishnaswami},
       title={The distribution of {$\nu (n)$} in the sieve of {E}ratosthenes},
        date={1982},
        ISSN={0033-5606},
     journal={Quart. J. Math. Oxford Ser. (2)},
      volume={33},
      number={130},
       pages={129\ndash 148},
  url={https://doi-org.proxy-tu.researchport.umd.edu/10.1093/qmath/33.2.129},
      review={\MR{657120}},
}

\bib{Bal90}{article}{
      author={Balazard, Michel},
       title={Unimodalit\'{e} de la distribution du nombre de diviseurs
  premiers d'un entier},
        date={1990},
        ISSN={0373-0956},
     journal={Ann. Inst. Fourier (Grenoble)},
      volume={40},
      number={2},
       pages={255\ndash 270},
         url={http://www.numdam.org/item?id=AIF_1990__40_2_255_0},
      review={\MR{1070828}},
}

\bib{BDN}{article}{
      author={Balazard, Michel},
      author={Delange, Hubert},
      author={Nicolas, Jean-Louis},
       title={Sur le nombre de facteurs premiers des entiers},
        date={1988},
        ISSN={0249-6291},
     journal={C. R. Acad. Sci. Paris S\'{e}r. I Math.},
      volume={306},
      number={13},
       pages={511\ndash 514},
      review={\MR{941613}},
}

\bib{dB51}{article}{
      author={de~Bruijn, N.~G.},
       title={On the number of positive integers {$\leq x$} and free of prime
  factors {$>y$}},
        date={1951},
     journal={Nederl. Acad. Wetensch. Proc. Ser. A.},
      volume={54},
       pages={50\ndash 60},
      review={\MR{0046375}},
}

\bib{dekoninck72}{article}{
      author={De~Koninck, Jean-Marie},
       title={On a class of arithmetical functions},
        date={1972},
        ISSN={0012-7094},
     journal={Duke Math. J.},
      volume={39},
       pages={807\ndash 818},
         url={http://projecteuclid.org/euclid.dmj/1077380588},
      review={\MR{311598}},
}

\bib{DKDO19}{article}{
      author={De~Koninck, Jean-Marie},
      author={Doyon, Nicolas},
      author={Ouellet, Vincent},
       title={The limit distribution of the middle prime factors of an
  integer},
        date={2019},
     journal={Integers},
      volume={19},
       pages={Paper No. A56, 23},
      review={\MR{4030537}},
}

\bib{DKK14}{article}{
      author={De~Koninck, Jean-Marie},
      author={K\'{a}tai, Imre},
       title={Normal numbers and the middle prime factor of an integer},
        date={2014},
        ISSN={0010-1354},
     journal={Colloq. Math.},
      volume={135},
      number={1},
       pages={69\ndash 77},
         url={https://doi-org.proxy-tu.researchport.umd.edu/10.4064/cm135-1-5},
      review={\MR{3215369}},
}

\bib{DKLuca13}{article}{
      author={De~Koninck, Jean-Marie},
      author={Luca, Florian},
       title={On the middle prime factor of an integer},
        date={2013},
     journal={J. Integer Seq.},
      volume={16},
      number={5},
       pages={Article 13.5.5, 10},
      review={\MR{3065334}},
}

\bib{Dickman}{article}{
      author={Dickman, Karl},
       title={On the frequency of numbers containing prime factors of a certain
  relative magnitude},
        date={1930},
     journal={Arkiv for matematik, astronomi och fysik},
      volume={22},
      number={10},
       pages={A\ndash 10},
}

\bib{DoyOue18}{article}{
      author={Doyon, Nicolas},
      author={Ouellet, Vincent},
       title={On the sum of the reciprocals of the middle prime factor counting
  multiplicity},
        date={2018},
        ISSN={0138-9491},
     journal={Ann. Univ. Sci. Budapest. Sect. Comput.},
      volume={47},
       pages={249\ndash 259},
      review={\MR{3849207}},
}

\bib{HT88}{book}{
      author={Hall, {R.\,R.}},
      author={Tenenbaum, G.},
       title={Divisors},
      series={Cambridge Tracts in Mathematics},
   publisher={Cambridge University Press, Cambridge},
        date={1988},
      volume={90},
}

\bib{harman}{book}{
      author={Harman, Glyn},
       title={Metric number theory},
      series={London Mathematical Society Monographs. New Series},
   publisher={The Clarendon Press, Oxford University Press, New York},
        date={1998},
      volume={18},
        ISBN={0-19-850083-1},
      review={\MR{1672558}},
}

\bib{KTP}{article}{
      author={Knuth, Donald~E.},
      author={Trabb~Pardo, Luis},
       title={Analysis of a simple factorization algorithm},
        date={1976/77},
        ISSN={0304-3975},
     journal={Theoret. Comput. Sci.},
      volume={3},
      number={3},
       pages={321\ndash 348},
         url={https://doi.org/10.1016/0304-3975(76)90050-5},
      review={\MR{498355}},
}

\bib{Lichtman}{article}{
      author={Lichtman, Jared~Duker},
       title={Mertens' prime product formula, dissected},
        date={2021},
     journal={Integers},
      volume={21A},
      number={Ron Graham Memorial Volume},
       pages={Paper No. A17, 15},
      review={\MR{4304373}},
}

\bib{MPS}{article}{
      author={McNew, Nathan},
      author={Pollack, Paul},
      author={Singha~Roy, Akash},
       title={Intermediate prime factors in specified subsets},
     journal={Monatsh. fur Math.},
       pages={15 pp.},
      status={to appear},
}

\bib{Norton}{article}{
      author={Norton, Karl~K.},
       title={On the number of restricted prime factors of an integer. {I}},
        date={1976},
        ISSN={0019-2082},
     journal={Illinois J. Math.},
      volume={20},
      number={4},
       pages={681\ndash 705},
  url={http://projecteuclid.org.proxy-tu.researchport.umd.edu/euclid.ijm/1256049659},
      review={\MR{419382}},
}

\bib{Ouellet17}{article}{
      author={Ouellet, Vincent},
       title={On the sum of the reciprocals of the middle prime factors of an
  integer},
        date={2017},
     journal={J. Integer Seq.},
      volume={20},
      number={10},
       pages={Art. 17.10.1, 16},
      review={\MR{3745200}},
}

\bib{ten99}{incollection}{
      author={Tenenbaum, G\'{e}rald},
       title={Crible d'\'{E}ratosth\`ene et mod\`ele de {K}ubilius},
        date={1999},
   booktitle={Number theory in progress, {V}ol. 2
  ({Z}akopane-{K}o\'{s}cielisko, 1997)},
   publisher={de Gruyter, Berlin},
       pages={1099\ndash 1129},
      review={\MR{1689563}},
}

\bib{tenenbaumAPNT}{book}{
      author={Tenenbaum, G\'{e}rald},
       title={Introduction to analytic and probabilistic number theory},
     edition={Third},
      series={Graduate Studies in Mathematics},
   publisher={American Mathematical Society, Providence, RI},
        date={2015},
      volume={163},
        ISBN={978-0-8218-9854-3},
         url={https://doi-org.proxy-tu.researchport.umd.edu/10.1090/gsm/163},
        note={Translated from the 2008 French edition by Patrick D. F. Ion},
      review={\MR{3363366}},
}

\end{biblist}
\end{bibdiv}
\bigskip 

\end{document}